\documentclass[11pt,reqno,a4paper]{amsart}      
\usepackage{amsmath,amssymb,amsthm,color,soul,enumerate,enumitem,cancel}
\usepackage{hyperref} 
\newtheorem{thm}{Theorem}
\newtheorem{thmx}{Theorem}

\newtheorem{prop}[thm]{Proposition}
\newtheorem{cor}[thm]{Corollary}
\newtheorem{claim}[thm]{Claim}
\newtheorem{lemma}[thm]{Lemma}
\newtheorem{fact}[thm]{Fact}
\newtheorem{question}[thm]{Question}
\newtheorem{problem}[thm]{Problem}
\newtheorem{conjecture}[thm]{Conjecture}
\theoremstyle{definition}
\newtheorem{defin}[thm]{Definition}
\theoremstyle{remark}
\newtheorem{remark}[thm]{Remark}

\newcommand{\en}{\mathbb N}
\newcommand{\setsep}{:\;}

\newcommand{\Rea}{\mathbb{R}}
\newcommand{\Nat}{\mathbb{N}}
\newcommand{\Rat}{\mathbb{Q}}
\newcommand{\Int}{\mathbb{Z}}
\newcommand{\Norm}{\|\cdot \|}
\newcommand{\dist}{\operatorname{dist}}

\newcommand{\dom}{\operatorname{dom}}

\newcommand{\id}{\operatorname{id}}
\newcommand{\Met}{\mathcal{M}}
\newcommand{\Banach}{\mathcal{B}}
\newcommand{\Corr}{\mathcal{R}}
\newcommand{\game}{\mathcal{G}}
\newcommand{\Span}{\operatorname{span}}

\newcommand{\Lip}{\operatorname{Lip}}

\usepackage{pgf}
\usepackage{tikz}
\usetikzlibrary{arrows,automata}
\usetikzlibrary{positioning}

\begin{document}

\title[Complexity of distances]{Complexity of distances: Theory of generalized analytic equivalence relations}

\author[M. C\' uth]{Marek C\'uth}
\author[M. Doucha]{Michal Doucha}
\author[O. Kurka]{Ond\v{r}ej Kurka}
\email{cuth@karlin.mff.cuni.cz}
\email{doucha@math.cas.cz}
\email{kurka.ondrej@seznam.cz}

\address[M.~C\' uth, O.~Kurka]{Charles University, Faculty of Mathematics and Physics, Department of Mathematical Analysis, Sokolovsk\'a 83, 186 75 Prague 8, Czech Republic}
\address[M.~Doucha, O.~Kurka]{Institute of Mathematics of the Czech Academy of Sciences, \v{Z}itn\'a 25, 115 67 Prague 1, Czech Republic}



\thanks{M. C\' uth was supported by Charles University Research program No. UNCE/SCI/023 and by the Research grant GA\v CR 17-04197Y. M. Doucha was supported by the GA\v CR projects 16-34860L and EXPRO 20-31529X, and RVO: 67985840. O. Kurka was supported by the Research grant GA\v CR 17-04197Y and by RVO: 67985840.}

\begin{abstract}
We generalize the notion of analytic/Borel equivalence relations, orbit equivalence relations, and Borel reductions between them to their continuous and quantitative counterparts: analytic/Borel pseudometrics, orbit pseudometrics, and Borel reductions between them. We motivate these concepts on examples and we set some basic general theory. We illustrate the new notion of reduction by showing that the Gromov-Hausdorff distance maintains the same complexity if it is defined on the class of all Polish metric spaces, spaces bounded from below, from above, and from both below and above. Then we show that $E_1$ is not reducible to equivalences induced by orbit pseudometrics, generalizing the seminal result of Kechris and Louveau. We answer in negative a question of Ben-Yaacov, Doucha, Nies, and Tsankov on whether balls in the Gromov-Hausdorff and Kadets distances are Borel.  In appendix, we provide new methods using games showing that the distance-zero classes in certain pseudometrics are Borel, extending the results of Ben Yaacov, Doucha, Nies, and Tsankov.

There is a complementary paper of the authors where reductions between the most common pseudometrics from functional analysis and metric geometry are provided.
\end{abstract}
\maketitle

\section*{Introduction}
One of the main active streams of the current descriptive set theory, often called \emph{invariant descriptive set theory}, is concerned with the study of definable equivalence relations on standard Borel spaces and reductions between them. This is a subject that has turned out to be very helpful in many fields of mathematics. Indeed, a common theme in mathematics is classification of some sort, which is in other words studying the isomorphism equivalence relation in some category and finding some effective reduction from that equivalence relation to another (isomorphism) relation which is simpler and more understood. Invariant descriptive set theory provides a general framework for such investigations and can be viewed as a general classification theory. We refer to \cite{gao} for a reference to this subject.

The complexity of the isomorphism relation in several major mathematical categories has been recently determined. Just to name a few, in \cite{Sabok} the isomorphism relation of separable $C^*$-algebras was reduced to the univeral orbit equivalence relation. The same complexity was shown for the homeomorphism relation of compact metrizable spaces in \cite{Zie} (thus in fact for the isomorphism relation of separable unital commutative $C^*$-algebras) and also for the linear isometry relation of separable Banach spaces in \cite{Melleray}. The complexities of the linear isomorphism relation for separable Banach spaces, resp. of the completely bounded isomorphism relation for separable operator spaces, have been, on the other hand, shown to be complete analytic in \cite{FLR}, resp. in \cite{ACKKLS}.
\medskip

It turns out however that in many areas of mathematics, especially in those working with `metric objects', such as functional analysis or metric geometry, it is often convenient and more accessible to replace the isomorphism relation with some approximations. These approximations usually come in the form of some metric or pseudometric which measures how close to being isomorphic two objects are. Prototypical examples are the Gromov-Haudorff distance between compact metric spaces introduced by Gromov (\cite{Gro}) which measures how close two compact metric space are to being isometric, or the Banach-Mazur distance between finite-dimensional Banach spaces which measures how close two finite-dimensional Banach spaces are to being linearly isometric. In both these examples, when two spaces have distance zero they are isometric, resp. linearly isometric. However, in most of the more complicated examples this is not the case and the studied distance is in fact only a pseudometric. Consequently, it induces an equivalence relation that in general does not coincide with the standard isomorphism relation in the corresponding category and it is worth studying by its own. This happens e.g. when one considers the Gromov-Hausdorff, resp. the Banach-Mazur distances on general (complete) metric spaces, resp. general Banach spaces. We note that nowadays we have examples of such distances in many areas of mathematics, e.g. in metric space theory (see \cite{Gro} or \cite{BBI} for several other distances on metric spaces), in Banach space theory we mention the Gromov-Hausdorff distance analogue for Banach spaces, the Kadets distance (see \cite{Kad}), or various distances introduced e.g. by Ostrovskii (see \cite{O} and \cite{o94}), in operator algebras (see the Kadison-Kastler distance defined in \cite{KadKas}), in the theory of graph limits (see various distances defined on graph limits, e.g. graphons, in \cite{Lov}), in measure theory (see a number of distances between measures in \cite{Gibbs}).

Our goal in this paper is to view these pseudometrics as generalized equivalence relations. This follows the research from \cite{BYDNT} where certain back-and-forth Borel equivalence relations, coming from metric Scott analysis, approximating the isomorphism relation on a class of countable structures were replaced by pseudometrics measuring how close to being isomorphic two metric objects are. Moreover, it is also in the spirit of the model theory for metric structures, which has been enjoying a lot of developments and applications recently, to generalize discrete notions by their continuous counterparts. See \cite{metriclogic} for an introduction to that subject. In the case of equivalence relations, the natural continuous counterpart is the notion of a pseudometric. The innovation in our paper comes from the idea to generalize the standard notion of Borel reducibility between definable equivalence relations to a Borel reducibility between definable pseudometrics. This notion arises naturally when proving reductions between equivalence relations that are induced by pseudometrics and noticing that such reductions are often quantitative and provide more information.
\medskip

Suppose we are given two standard Borel spaces of metric structures, each equipped with some definable pseudometric, and we want to effectively reduce the first pseudometric to the other. The natural choice, generalizing the standard theory, is that the reduction be Borel. However, now it must continuously preserve the pseudometric in some sense. In any case, it should be a Borel reduction between the equivalence relations, in the standard theory, that are induced by these two pseudometrics. Some obvious choices could be that the reduction is isometric, or bi-Lipschitz, which seems to be too strong though. The right notion that most often appears naturally in our considerations is that the reduction is a uniformly continuous embedding (which is not in general injective when working with pseudometrics). This is moreover sufficient for our applications. We call such a Borel reduction \emph{Borel-uniformly continuous reduction}.

It can be certainly said that most of the most important analytic equivalence relations come from Polish group actions on standard Borel spaces. Such equivalence relations are called \emph{orbit equivalence relations}. Another main innovation of the paper is to suggest a continuous generalization of this class of equivalence relations. That is something we call \emph{orbit pseudometrics}. Under some natural restrictions, these seem to form a very interesting class of analytic pseudometrics. Indeed, many of the interesting and non-trivial analytic pseudometrics are bi-reducible with such orbit pseudometrics. The feature that makes them particularly interesting is that the equivalence relation $E_1$ is not Borel reducible to equivalence relations induced by these orbit pseudometrics, which we prove in this paper.
\medskip

The goal of this paper is first to define these new notions and make some first steps in establishing general theory around them. Second, we generalize the celebrated result of Kechris and Louveau about non-reducibility of the equivalence relation $E_1$ in the setting of orbit pseudometrics and using our methods we answer Question 8.4 from \cite{BYDNT} about Borelness of some of these pseudometrics. We provide also several reductions between the Gromov-Hausdorff distance defined on different classes of metric spaces. However, we have a complementary paper \cite{CDKpart2} where we concentrate reductions between the most common pseudometrics appearing in functional analysis and metric geometry.

We summarize our main results below. We show several reductions between the Gromov-Hausdorff distance defined on various classes of metric spaces.
\begin{thmx}\label{thm:introduction1}
The following restrictions of the Gromov-Hausdorff distance are mutually Borel-uniformly continuous bi-reducible: Gromov-Hausdorff distance defined on all Polish metric spaces, restricted to Polish metric spaces bounded from above, from below, from both above and below (see Theorem~\ref{thm:gh}).
\end{thmx}

As mentioned above, most of the interesting pseudometrics from functional analysis and metric geometry (including the Gromov-Hausdorff distance), and which are investigated in \cite{CDKpart2}, actually belong to a special class of analytic pseudometrics, which we call \emph{CTR orbit pseudometrics}. We refer the reader to Section \ref{section:orbitPseudometrics} for a precise definition. CTR orbit pseudometrics naturally generalize orbit equivalence relations and we are able to extend the well known result of Kechris and Louveau (from \cite{KeLo97}) on non-reducibility of the equivalence relation $E_1$ into them.
\begin{thmx}
\begin{enumerate}
\item Gromov-Hausdorff distance (and therefore all other distances bi-reducible with it, see \cite{CDKpart2}) is Borel-uniformly continuous bi-reducible with a CTR orbit pseudometric (see Theorem~\ref{thm:ghCTR}).
\item The equivalence relation $E_1$ is not Borel reducible to any equivalence relation given by a CTR orbit pseudometric (see Theorem~\ref{thm:E1notreducible}).
\end{enumerate}
\end{thmx}
 Next, we answer Question 8.4 from \cite{BYDNT} by proving the following. The Kadets distance will be defined in the next section.

\begin{thmx}
Let $\rho$ be any pseudometric to which the Kadets distance is Borel-uniformly continuous reducible (see \cite{CDKpart2} for a large list of examples). Then there is an element $A$ from the domain of $\rho$ such that the function $\rho(A,\cdot)$ is not Borel (see Theorem~\ref{thm:distancenotBorel}).
\end{thmx}

Finally, in an appendix, we extend the results from \cite{BYDNT} where it was shown that the equivalence classes of the Gromov-Hausdorff and Kadets distances are Borel.


\bigskip

Let us note that each pseudometric $\rho$ on a set $X$ induces also a cruder equivalence relation, which we denote here by $E^\rho$, where for $x,y\in X$ we set $x E^\rho y$ if and only if $\rho(x,y)<\infty$. The complexity of such relations, for the natural pseudometrics from functional analysis and metric geometry, has been studied recently rather extensively. For example, for the Lipschitz and Banach-Mazur distances, the complexity of such equivalences was determined in \cite{FLR}, where it was shown that they are complete analytic equivalence relations. For the Gromov-Hausdorff distance, it was studied recently in \cite{AC}, where it was shown that this equivalence is not Borel reducible to an orbit equivalence relation. A plausible conjecture is that for all analytic pseudometrics $\rho$ bi-reducible with the Gromov-Hausdorff distance, the equivalence relation $E^\rho$ is complete analytic.

There are also certain analogies and differences between equivalence relations of the form $E_\rho$ and $E^\rho$, for some analytic pseudometric $\rho$, that we would like to point out here and which we plan to investigate later. Suppose we are given analytic pseudometrics $\rho_1,\rho_2$, and suppose we are focused on the equivalence relations $E_{\rho_1}$ and $E_{\rho_2}$ and Borel reducibility between them. The natural thing to do, which is thoroughly investigated in this paper and in \cite{CDKpart2}, is then to make the reduction quantitative in the sense that the reduction is actually a uniformly continuous embedding. When on the other hand working with the equivalence relations $E^{\rho_1}$ and $E^{\rho_2}$, a different situation emerges. Two elements of $E^{\rho_1}$ are equivalent if they merely have a finite, but arbitrarily large, distance in $\rho_1$. Thus quantitative reduction between $E^{\rho_1}$ and $E^{\rho_2}$ probably should ignore the small scale and it would be overprecise to require it to be uniformly continuous. Instead, it should, in the spirit of large scale geometry, be concerned only with the large distances, and therefore the appropriate notion seems to be of coarse reduction between $\rho_1$ and $\rho_2$. This will be investigated in future research.

Moreover, it would be an interesting project on its own to investigate which analytic equivalence relation from practice are naturally of the form $E_\rho$ or $E^\rho$ for some non-discrete non-trivial pseudometric.

\bigskip
The paper is organized as follows. In Section~\ref{section:pseudometrics} we recall several basic notions from descriptive set theory, introduce our new notions such as Borel reducibility between analytic pseudometrics and we provide several examples. In Section~\ref{section:gh}, we prove several reductions concerning the Gromov-Hausdorff distance that will be useful in the subsequent section and which illustrates the new notion of reducibility. In Section~\ref{section:orbitPseudometrics}, we introduce the continuous version of orbit equivalence relations, the CTR orbit pseudometrics, and we prove that the equivalence relation $E_1$ is not Borel reducible to the equivalence relations induced by them. In Section~\ref{section:notBoreldist}, we prove that $\rho$-balls are in general not Borel for any pseudometric $\rho$ to which the Kadets distance is reducible. Finally, in Section~\ref{section:problems}, we comment on our results and we present directions for further research. The paper is concluded with Appendix~\ref{sectionGames}  where we play certain metric games which provide an alternative, and probably more general, way to the methods from \cite{BYDNT} how to show that these pseudometrics have Borel classes of equivalence.  
\section{Analytic pseudometrics and reductions between them}\label{section:pseudometrics}
In this section, we introduce several new concepts that generalize the standard theory of Borel/analytic equivalence relations on Polish and standard Borel spaces and the reductions between them.

Recall that a Borel (analytic) equivalence relation $E$ on a Polish (or more generally standard Borel) space $X$ is a subset $E\subseteq X^2$ that is an equivalence relation and is a Borel (analytic) subset of the space $X^2$. If $E$ and $F$ are two equivalence relations, Borel or analytic, on spaces $X$, resp. $Y$, then we say that $E$ is Borel reducible to $F$, $E\leq_B F$ in symbols, if there exists a Borel function $f: X\rightarrow Y$ such that for every $x,y\in X$ we have $x E y$ if and only if $f(x) F f(y)$. Our reference for invariant descriptive set theory dealing with these notions is \cite{gao}.

Below we introduce the notions of Borel/analytic pseudometrics, generalizing the Borel/analytic equivalence relations, and the Borel reductions between them. We provide few general results about them. Reductions between the most important pseudodistances from metric geometry and functional analysis are the content of the complementary article \cite{CDKpart2}.

An important part of this section is also a list of such examples that demonstrates there is enough space for further investigations in this area. 

\subsection{Analytic pseudometrics and notions of a reduction}

\begin{defin}
Let $X$ be a standard Borel space. A pseudometric $\rho:X\times X\to [0,\infty]$ is called \emph{an analytic pseudometric}, resp. \emph{a Borel pseudometric}, if for every $r > 0$ the set $\{(x,y)\in X^2\setsep \rho(x,y)<r\}$ is analytic, resp. Borel.

Note that in the Borel case, this is equivalent to saying that $\rho$ is a Borel function. We emphasize that pseudometrics in our definition may attain $\infty$ as a value.
\end{defin}
The trivial examples of analytic (or Borel) pseudometrics come from analytic equivalence relations. Conversely, every analytic pseudometric induces an analytic equivalence relation.

\begin{defin}
Let $X$ be a set, $\rho$ a pseudometric on $X$ and $E$ an equivalence relation on $X$. By $E_\rho$ we denote the equivalence relation on $X$ defined by $E_\rho := \{(x,y)\in X\times X\setsep \rho(x,y)=0\}$. By $\rho_E$ we denote the pseudometric on $X$ with values in $\{0,1\}$ defined by $\rho_E(x,y) = 0$ iff $(x,y)\in E$.
\end{defin}

It is easy to check the following.

\begin{fact}
For every analytic (Borel) pseudometric $\rho$ on a standard Borel space $X$, the induced equivalence relation $E_\rho$ on $X$ is analytic (Borel). Conversely, for every analytic (Borel) equivalence relation $E$, the pseudometric $\rho_E$ is analytic (Borel).
\end{fact}

Now we introduce the main new definition of the paper.

\begin{defin}\label{defin:Borel-UnifRedukce}
Let $X$, resp. $Y$ be standard Borel spaces and let $\rho_X$, resp. $\rho_Y$ be analytic pseudometrics on $X$, resp. on $Y$. We say that $\rho_X$ is \emph{Borel-uniformly continuous reducible} to $\rho_Y$, $\rho_X\leq_{B,u} \rho_Y$ in symbols, if there exists a Borel function $f: X\rightarrow Y$ such that, for every $\varepsilon>0$ there are $\delta_X>0$ and $\delta_Y>0$ satisfying
\[
\forall x,y\in X:\quad \rho_X(x,y)<\delta_X\Rightarrow \rho_Y(f(x),f(y))<\varepsilon
\]
and
\[
\forall x,y\in X:\quad \rho_Y(f(x),f(y))<\delta_Y\Rightarrow \rho_X(x,y)<\varepsilon.
\]
In this case we say that $f$ is a \emph{Borel-uniformly continuous reduction}. If $\rho_X\leq_{B,u} \rho_Y$ and $\rho_Y\leq_{B,u} \rho_X$, we say that $\rho_X$ is \emph{Borel-uniformly continuous bi-reducible} with $\rho_Y$ and write $\rho_X\sim_{B,u}\rho_Y$.

Moreover, if $f$ is injective we say it is an \emph{injective Borel-uniformly continuous reduction}.

If $f$ is an isometry from the pseudometric space $(X,\rho_X)$ into $(Y,\rho_Y)$, we say it is a \emph{Borel-isometric reduction}.

If there are $\varepsilon > 0$ and $C>0$ such that for every $x,y\in X$ we have
\[\begin{split}
	\rho_X(x,y) < \varepsilon & \implies \rho_Y(f(x),f(y))\leq C\rho_X(x,y)\\
    \text{and}\quad \rho_Y(f(x),f(y)) < \varepsilon & \implies \rho_X(x,y)\leq C\rho_Y(f(x),f(y)),
\end{split}\]
we say that $f$ is a \emph{Borel-Lipschitz on small distances reduction}.
\end{defin}

The definition of a Borel-uniformly continuous reduction seems to be the most useful one in the sense that it is strong enough for our applications, yet it naturally arises in our examples. Sometimes we are able to demonstrate the reducibility between some pseudometrics by maps with stronger properties and this is the reason why we mentioned the remaining notions above.

\begin{remark}
Note that in particular $\rho_X\leq_{B,u} \rho_Y$ implies the reducibility between the corresponding equivalence relations, i.e. $E_{\rho_X}\leq_B E_{\rho_Y}$ and the same Borel function $f$ is a witness. So Borel-uniformly continuous reducibility between pseudometrics is a stronger notion than the Borel reducibility between the corresponding equivalence relations.

Moreover, $ E_{X} \leq_B E_{Y} $ is the same as $ \rho_{E_{X}} \leq_{B,u} \rho_{E_{Y}} $. So Borel-uniformly continuous reducibility between pseudometrics is a generalization of the notion of Borel reducibility between equivalence relations.
\end{remark}

\begin{thm}\label{thm:universal_pseudometric}
There exists a universal analytic pseudometric. That is, for any analytic pseudometric there is a Borel-isometric reduction into the universal one.
\end{thm}
\begin{proof}
Let $\mathcal{U}\subseteq \Nat^\Nat\times ((\Nat^\Nat)^2)\times \Rat^+$ be a universal analytic subset for $(\Nat^\Nat)^2\times \Rat^+$ (see e.g. \cite[Theorem 14.2]{Ke} for the existence). That is, for every analytic subset $A\subseteq (\Nat^\Nat)^2\times \Rat^+$ there exists $u\in \Nat^\Nat$ such that $A=\mathcal{U}_u$. For every $p\in\Rat^+$ we set 
\[\begin{split}
U_p=\Big\{ & (a,x,y)\in (\Nat^\Nat)^3\setsep \exists q_1,\ldots,q_n\in\Rat^+\;(\sum_{i\leq n} q_i<p),\\
& \exists z_0,z_1,\ldots,z_n\in\Nat^\Nat, z_0=x, z_n = y, \;(\forall i\leq n\; (a,z_{i-1},z_{i},q_{i})\in\mathcal{U})\Big\}.
\end{split}\]

It is easy to check that $U_p$ is analytic. We define a pseudometric $\rho$ on $(\Nat^\Nat)^2$ as follows. For $(a,x),(a',y)\in(\Nat^\Nat)^2$ we set $\rho((a,x),(a,x))=0$ and $\rho((a,x),(a',y))=\infty$ if $a\neq a'$. Otherwise we set 
\[
\rho((a,x),(a,y))=\inf\big\{p\in\Rat^+\setsep (a,x,y)\in U_p \text{ and } (a,y,x)\in U_p\big\}.
\]
It is easy to check that $\rho$ is a pseudometric. It is also clear that for every $r>0$ the set $\{((z,x),(z',y))\in (\Nat^\Nat)^4\setsep \rho((z,x),(z',y))<r\}$ is analytic since this set is equal to
\[
\bigcup_{p<r, p\in\Rat^+} \big\{((z,x),(z',y))\in (\Nat^\Nat)^4\setsep z=z', (z,x,y)\in U_p, (z,y,x)\in U_p\big\}.
\]

Now let $\sigma$ be an arbitrary analytic pseudometric on some standard Borel space, which we may assume without loss of generality is equal to $\Nat^\Nat$. For each $p\in\Rat^+$ let $A_p$ be the analytic set $\{(x,y)\in (\Nat^\Nat)^2\setsep \sigma(x,y)<p\}$. The set $\mathcal{A}=\bigcup_{p\in\Rat^+} A_p\times\{p\}$ is analytic. Therefore there exists $u\in\Nat^\Nat$ such that $\mathcal{U}_u=\mathcal{A}$. Now, we easily verify that the mapping $(\Nat^\Nat,\sigma)\ni y\mapsto (u,y)\in ((\Nat^\Nat)^2,\rho)$ is an isometry and so it is the desired reduction.
\end{proof}

\subsection{Examples}
There are many natural examples of distances between various structures in mathematics that give rise to analytic pseudometrics on standard Borel spaces. Many of those are handled in a greater detail in our paper \cite{CDKpart2}.

First, let us formalize the class of all the separable metric and Banach spaces as it is done in \cite[Section 1.1]{CDKpart2}.

\begin{defin}
By $\Met$ we denote the space of all metrics on $\Nat$. This gives $\Met$ a Polish topology inherited from $\Rea^{\Nat\times\Nat}$.

If $p$ and $q$ are positive real numbers, by $\Met_p$, $\Met^q$ and $\Met_p^q$ respectively, we denote the space of metrics with values in $\{0\}\cup [p,\infty)$, $[0,q]$, and $\{0\}\cup[p,q]$ (assuming that $p<q$), respectively.
\end{defin}
\begin{remark}Every $f\in\Met$ is then a  code for Polish metric space $M_f$ which is the completion of $(\en,f)$.  Hence, in this sense we may refer to the set $\Met$ as to the standard Borel space of all infinite Polish metric spaces. This approach was used for the first time by Vershik \cite{ver} and further e.g. in \cite{cle}, see also \cite[page 324]{gao}.
\end{remark}

\begin{defin}\label{defin:spaceOfBanachSpaces}
Let us denote by $V$ the vector space over $\Rat$ of all finitely
supported sequences of rational numbers, that is, the unique infinite-dimensional vector space over $\Rat$ with a countable Hamel basis $(e_n)_{n\in\Nat}$. By $\Banach_0$ we denote the space of all norms on the vector space $V$. This gives $\Banach_0$ a Polish topology inherited from $\Rea^V$. We shall consider only those norms for which its canonical extension to the real vector space $c_{00}$ is still a norm; that is, norms for which the elements $(e_n)_n$ are not only $\Rat$-linearly independent, but also $\Rea$-linearly independent. Let us denote the subset of such norms by $\Banach$. It is a $G_\delta$ subset of $\Banach_0$, see \cite[p. 4]{CDKpart2}.
\end{defin}

\begin{remark} Each norm $\nu\in\Banach$ is then a code for an infinite-dimensional Banach space $X_\nu$ which is the completion of $(V,\nu)$. The completion is naturally a complete normed space over $\Rea$. This is the same as taking the canonical extension of $\nu$ to $c_{00}$ and then taking the completion.

Hence, we may refer to the set $\Banach$ as to the standard Borel space of all infinite-dimensional separable Banach spaces. This approach seems to be invented by the authors, we refer to \cite[Section 1.1]{CDKpart2} for a more detailed discussion and connection to other possible codings used by other authors.
\end{remark}

Now, let us present some natural examples of analytic pseudometrics.

\smallskip

\begin{enumerate}[leftmargin=0cm,itemindent=.5cm,start=1,label={\bfseries \arabic*. }]
	\item {\bf Gromov-Hausdorff distance}\label{ex:gh1} Let $(M,d_M)$ be a metric space and $A,B\subseteq M$ two non-empty subsets. The \emph{Hausdorff distance} between $A$ and $B$ in $M$, $\rho_H^M(A,B)$, is defined as
\[
\max\Big\{\sup_{a\in A} d_M(a,B),\sup_{b\in B} d_M(b,A)\Big\},
\]
where for an element $a\in M$ and a subset $B\subseteq M$, $d_M(a,B)=\inf_{b\in B} d_M(a,b)$. The \emph{Gromov-Hausdorff distance} between metric spaces $M$ and $N$, $\rho_{GH}(M,N)$, is defined as the infimum of the Hausdorff distances of their isometric copies contained in a single metric space, that is
\[
\rho_{GH}(M,N)=\inf_{\substack{\iota_M:M\hookrightarrow X\\ \iota_N: N\hookrightarrow X}} \rho_H^X(\iota_M(M),\iota_N(N)),
\]
where $\iota_M$ and $\iota_N$ are isometric embeddings into a metric space $X$. We refer to \cite[Section 7.3]{BBI} for some details concerning this distance.

Equip the Polish space $\Met$ with the Gromov-Hausdorff distance $\rho_{GH}$, that is, $\rho_{GH}(f,g)$ is the Gromov-Hausdorff distance between $(\en,f)$ and $(\en,g)$, which is easily seen to be equal to the Gromov-Hausdorff distance between their completions $M_f$ and $M_g$. We can also consider the pseudometric $\rho_{GH}$ on the space $\Banach$ of codes for separable Banach spaces, denoted there by $\rho_{GH}^\Banach$. Note that for Banach spaces $X$ and $Y$, $\rho_{GH}^\Banach(X,Y)$ is defined as the Gromov-Hausdorff distance of the unit balls $B_X$ and $B_Y$ (see e.g. the introduction in \cite{KalOst}). Both $\rho_{GH}$ and $\rho_{GH}^\Banach$ are analytic pseudometrics, see Section~\ref{section:gh} for more details.

\smallskip

	\item {\bf Kadets distance} 
	The \emph{Kadets distance} between Banach spaces $X$ and $Y$, $\rho_K(X,Y)$, is defined as the infimum of the Hausdorff distances of their unit balls over all isometric linear embeddings of $X$ and $Y$ into a common Banach space $Z$. That is \[
\rho_K(X,Y)=\inf_{\substack{\iota_X: X\hookrightarrow Z\\ \iota_Y: Y\hookrightarrow Z}} \rho_H^Z(\iota_X(B_X),\iota_Y(B_Y)),
\]
where $\iota_X$ and $\iota_Y$ are linear isometric embeddings into a Banach space $Z$. We refer the interested reader e.g. to \cite{KalOst} for some more details concerning this distance.

Equip the Polish space $\Banach$ with the Kadets distance $\rho_K$, that is, $\rho_K(\mu,\nu)$ is the Kadets distance between Banach spaces $X_\mu$ and $X_\nu$. It is true (even thought not trivial) that $\rho_K$ is analytic pseudometric, see \cite[Section 1.3]{CDKpart2} for more details.

\smallskip
    
    \item {\bf Lipschitz distance} The \emph{Lipschitz distance} between metric spaces is defined as
\[
\rho_L(M,N) = \inf\big\{\log\max\{\Lip(T),\Lip(T^{-1})\}\setsep T:M\rightarrow N\text{ is bi-Lipschitz bijection}\big\},
\]
where 
\[
\Lip(T)=\sup_{m\neq n\in M}\frac{d_N(T(m),T(n))}{d_M(m,n)}
\]
is the Lipschitz norm of $T$. This notion of the Lipschitz distance is from \cite[Definition 7.2.1]{BBI}, equivalent notions are considered e.g. in \cite[Definition 3.1]{Gro} in the setting of metric spaces and in \cite{DuKa} in the setting of Banach spaces. We refer reader to \cite[Section 1.2.3]{CDKpart2} for some more detailed comments.
    
    Equip the Polish spaces $\Met$ and $\Banach$ with the Lipschitz distance $\rho_L$, that is for $d,p\in\Met$ and $\mu,\nu\in\Banach$ by $\rho_L(d,p) = \rho_L(M_d,M_p)$ and $\rho_L(\mu,\nu) = \rho_L(X_\mu,X_\nu)$, respectively. Whenever we consider the pseudometric $\rho_L$ on $\Banach$ and we want to emphasize it, we write $\rho_L^\Banach$ instead of just $\rho_L$. Then both $\rho_L$ and $\rho_L^\Banach$ are analytic pseudometrics, see \cite[Section 1.3]{CDKpart2} for more details.

\smallskip

 \item {\bf Banach-Mazur distance} The (logarithmic) \emph{Banach-Mazur distance} between Banach spaces $X$ and $Y$ is defined as
\[
\rho_{BM}(X,Y) = \inf\big\{\log \|T\|\|T^{-1}\|\setsep T:X\rightarrow Y\text{ is a linear isomorphism}\big\}.
\]
 
 Equip the Polish space $\Banach$ by the Banach-Mazur distance $\rho_{BM}$, that is, $\rho_{BM}(\mu,\nu) = \rho_{BM}(X_\mu,X_\nu)$. We refer the interested reader to \cite[Sections 1.2.4 and 1.3]{CDKpart2} for more details including the proof that $\rho_{BM}$ is analytic.

\smallskip

\item {\bf Hausdorff-Lipschitz and net distances} The \emph{Hausdorff-Lipschitz distance} between metric spaces $M$ and $N$ is defined as 
\[
\rho_{HL}(M,N)=\inf \big\{\rho_{GH}(M,M')+\rho_L(M',N')+\rho_{GH}(N',N)\setsep M',N'\text{ metric spaces}\big\}.
\]
The Hausdorff-Lipschitz distance corresponds to the notion of \emph{quasi-isometry} or \emph{coarse Lipschitz equivalence}, because for metric spaces $M$ and $N$ we have $\rho_{HL}(M,N)<\infty$ if and only if the spaces $M$ and $N$ are quasi-isometric, or coarse Lipschitz equivalent (see e.g. \cite[Section 8.3]{BBI} for further information). Note also that the Hausdorff distance between Banach spaces equals the \emph{net} distance considered in \cite{DuKa}, see \cite[Proposition 27]{CDKpart2}.

Equip the Polish spaces $\Met$ and $\Banach$ with the Hausdorff-Lipschitz distance $\rho_{HL}$ between the corresponding metric/Banach spaces. Then $\rho_{HL}$ is analytic, see \cite[Section 1.3]{CDKpart2}.

\smallskip

\item {\bf Uniform distance} Let $X$ and $Y$ be Banach spaces. If $u:X\to Y$  is uniformly continuous, we put
\[
	\Lip_{\infty}u:= \inf_{\eta>0}\;\sup\left\{\frac{\|u(x)-u(y)\|}{\|x-y\|}\setsep \|x-y\|\geq \eta\right\}.
\]
The \emph{uniform distance} between $X$ and $Y$, see e.g. \cite{DuKa}, is defined as
\[
\rho_U(X,Y) = \inf \big\{\log ((\Lip_{\infty}u)(\Lip_{\infty}{u^{-1}}))\setsep u:X\to Y \text{ is uniform homeomorphism} \big\}.
\]

Equip the Polish space $\Banach$ with the uniform distance $\rho_{U}$ between corresponding Banach spaces. We refer the interested reader to \cite[Sections 1.2.6 and 1.3]{CDKpart2} for more details including the proof that $\rho_{U}$ is analytic.

\smallskip

\item {\bf Completely bounded Banach-Mazur distance} Recall that an operator space is a closed linear subspace of a C*-algebra. The natural type of a morphism between operator spaces is a \emph{completely bounded} isomorphism (cb-isomorphism). In \cite{Pi95}, Pisier introduced the \emph{Banach-Mazur cb-distance} between two operator spaces $E$ and $F$:
\[
\rho_{CB}(E,F)= \inf\big\{\log \|u\|_{cb}\|u^{-1}\|_{cb}\setsep u:E\rightarrow F\text{ is a cb-isomorphism}\big\},
\]
where $\|u\|_{cb}$ is the completely bounded norm of $u$. For a background on completely bounded maps and operator spaces the reader is referred to \cite{Pis}. The standard Borel space of operator spaces was considered and described in \cite[Section 2.3]{ACKKLS}. We leave to the reader to verify that $\rho_{CB}$ is an analytic distance.
\end{enumerate}
\bigskip

For more examples from e.g. Banach space theory we refer the reader to articles \cite{O} and \cite{o94} of Ostrovskii where various distances between subspaces of a given Banach space are considered. See for example the \emph{Kadets path distance} in \cite{O} or the \emph{operator opening distance} in \cite{o94}. A good source of examples is also the Encyclopedia of distances \cite{Ency}.

We shall also see more examples in Section~\ref{section:orbitPseudometrics}.

\section{Complexity of the Gromov-Hausdorff distance}\label{section:gh}

In this section we show that the Gromov-Hausdorff distance (more precisely, the pseudometric $\rho_{GH}$ defined on the Polish space $\Met$ from Example~\ref{ex:gh1} in Section~\ref{section:pseudometrics}) is analytic pseudometric and we prove the following reducibility result.

\begin{thm}\label{thm:gh}Let $0<p<q$.
The following pseudometrics are mutually Borel-uniformly continuous bi-reducible: $\rho_{GH}$, $\rho_{GH}\upharpoonright \Met_p$, $\rho_{GH}\upharpoonright \Met^q$,  $\rho_{GH}\upharpoonright \Met_p^q$.
\end{thm}

The proof of Theorem~\ref{thm:gh} follows immediately from Corollary~\ref{cor:ReductionGHMettoMet^q_p}. It will be further used in order to show that $E_1$ is not Borel reducible to $E_{\rho_{GH}}$, see Theorem~\ref{thm:ghE1}.

Most of the examples mentioned in Section~\ref{section:pseudometrics} give rise to an analytic pseudometric which is Borel-uniformly continuous bi-redicible with $\rho_{GH}$. This is the topic handled in the complementary paper \cite{CDKpart2}, the results contained therein and in Theorem~\ref{thm:gh} are summarized as follows.

\begin{thm}\label{thm:intro1}
\begin{enumerate}
\item\label{thm:intro1:(1)} The following pseudometrics are mutually Borel-uniformly continuous bi-reducible: the {\bf Gromov-Hausdorff distance} when restricted to Polish metric spaces, to metric spaces bounded from above, from below, from both above and below, to Banach spaces; the {\bf Banach-Mazur distance} on Banach spaces, the {\bf Lipschitz distance} on Polish metric spaces and Banach spaces; the {\bf Kadets distance} on Banach spaces; the {\bf Hausdorff-Lipschitz distance} on Polish meric spaces; the {\bf net distance}  on Banach spaces.
\item The pseudometrics above are Borel-uniformly continuous reducible to the {\bf uniform} distance on Banach spaces.
\end{enumerate}
\end{thm}

\subsection{Preliminary results concerning the Gromov-Hausdorff distance}

Let $A$ and $B$ be two sets. A \emph{correspondence} between $A$ and $B$ is a binary relation $\Corr\subseteq A\times B$ such that for every $a\in A$ there is $b\in B$ such that $a\Corr b$, and for every $b\in B$ there is $a\in A$ such that $a\Corr b$.

\begin{fact}[see e.g. Theorem 7.3.25. in \cite{BBI}]\label{fact:GHbyCorrespondences}
Let $M$ and $N$ be two metric spaces. For every $r>0$ we have $\rho_{GH}(M,N)< r$ if and only if there exists a correspondence $\Corr$ between $M$ and $N$ such that $\sup |d_M(m,m')-d_N(n,n')|< 2r$, where the supremum is taken over all $m,m'\in M$ and $n,n'\in N$ with $m\Corr n$ and $m'\Corr n'$.
\end{fact}

It is easier to work with bijections instead of correspondences. One may wonder in which situations we may do so. Let us define the corresponding concept and prove some results in this direction. Those will be used later.

\begin{defin}
By $S_{\infty}$ we denote the set of all bijections from $\Nat$ to $\Nat$. For two metrics on natural numbers $ f, g \in \Met$ and $\varepsilon>0$, we consider the relation
\[
f \simeq_\varepsilon g \quad \Leftrightarrow \quad \exists \pi \in S_{\infty} \, \forall \{ n, m \} \in [\mathbb{N}]^{2} : |f(\pi(n), \pi(m)) - g(n, m)| \leq \varepsilon.
\]
We write $f\simeq g$ if $f\simeq_\varepsilon g$ for every $\varepsilon>0$.
\end{defin}

\begin{lemma}\label{lem:lehciImplikace}
For any two metrics on natural numbers $f, g\in \Met$ and any $\varepsilon>0$ we have $\rho_{GH}(f,g)\leq \varepsilon$ whenever $f\simeq_{2\varepsilon} g$.
\end{lemma}
\begin{proof}
Suppose that $f\simeq_{2\varepsilon} g$. The permutation $\pi\in S_\infty$ witnessing that $f\simeq_{2\varepsilon} g$ induces a correspondence $\Corr$ between the metric spaces $(\Nat,f)$ and $(\Nat, g)$ which, by Fact \ref{fact:GHbyCorrespondences}, shows that $\rho_{GH}(f,g)\leq \varepsilon$.
\end{proof}

\begin{lemma}\label{lem:GHequivalenceUniformlyDiscrete}
Let $p>0$ be a real number. For any two metrics on natural numbers $f, g\in \Met_p$ we have $\rho_{GH}(f,g)=\inf\{r\setsep f\simeq_{2r} g\}$ provided that $\rho_{GH}(f,g)< p/2$.
\end{lemma}
\begin{proof}
By Lemma~\ref{lem:lehciImplikace}, $\rho_{GH}(f,g)\leq r$ whenever $f\simeq_{2r} g$. Conversely, suppose that $\rho_{GH}(f,g)< r$, where $r<p/2$. By Fact \ref{fact:GHbyCorrespondences}, there exists a correspondence $\Corr$ between the metric spaces $(\Nat,f)$ and $(\Nat, g)$ such that $\sup |f(m,m')-g(n,n')|<2r$, where the supremum is over all $m,m',n,n'\in\Nat$ with $m\Corr n$ and $m'\Corr n'$. We claim that $\Corr$ is the graph of some permutation $\pi\in S_\infty$. That will, by definition of $\simeq_{2r}$, show that $f\simeq_{2r} g$. Suppose that $\Corr$ is not such graph. Say e.g. that for some $m$ there are $n\neq n'$ such that $m\Corr n$ and $m\Corr n'$. Then we have $g(n,n')=|f(m,m)-g(n,n')|<2r<p$, which contradicts that $g\in \Met_p$. Analogously we can show that for no $m\neq m'$ there is $n$ such that $m\Corr n$ and $m'\Corr n$.
\end{proof}

Finally, let us show that $\rho_{GH}$ is analytic.

\begin{prop}
The pseudometrics $\rho_{GH}:\Met\times\Met\to [0,\infty]$ and $\rho_{GH}^\Banach:\Banach\times\Banach\to [0,\infty]$ are analytic.
\end{prop}
\begin{proof}
We provide the proof for $\rho_{GH}$, the proof for $\rho_{GH}^\Banach$ is analogous and therefore omitted.
Fix some $r>0$. We claim that the set $D_r=\{(d,p)\in\Met^2\setsep \rho_{GH}(d,p)<r\}$ is analytic. Note that by Fact \ref{fact:GHbyCorrespondences}, $(d,p)\in D_r$ if and only if there exist a correspondence $\Corr\subseteq \Nat\times\Nat$ and $k\in\Nat$ such that $\forall i,j,m,n\in\Nat\;(i\Corr j, m\Corr n\Rightarrow |d(i,m)-p(j,n)|\leq 2r-1/k)$. Moreover, it is easy to see that the set 
\[\begin{split}
E_r =\bigcup_{k\in\Nat}\Big\{&(\Corr,d,p)\in \mathcal{P}(\Nat\times\Nat)\times \Met^2\setsep \forall i,j,m,n\in\Nat\;\\
&\big(i\Corr j, m\Corr n\Rightarrow |d(i,m)-p(j,n)|\leq 2r-1/k\big)\Big\}
\end{split}\]
is Borel for every $r > 0$. Since one can view the space of all correspondences $\mathcal{C}$ as a $G_\delta$ subspace of $\mathcal{P}(\Nat\times\Nat)$ and $(p,d)\in D_r$ if and only if $\exists \Corr\in \mathcal{C}\; ((\Corr,p,d)\in E_r)$, we get that $D_r$ is an analytic subset of $\Met^2$.
\end{proof}

\subsection{Reductions} In this subsection we prove our reducibility results leading to the proof of Theorem~\ref{thm:gh}.

\begin{thm}\label{thm:ReductionGHMet_5toMet^3}
For every positive real numbers $p$, $q$, there is an injective Borel-uniformly continuous reduction from $\rho_{GH}$ on $\Met_p$ to $\rho_{GH}$ on $\Met^q$.

Moreover, the reduction is not only Borel-uniformly continuous, but also Borel-Lipschitz on small distances.
\end{thm}
\begin{proof}
First, note that for every positive real numbers $p$, $q$ there is an injective Borel-uniformly continuous (and even Borel-Lipschitz) reduction from $\rho_{GH}$ on $\Met_p$ to $\rho_{GH}$ on $\Met_5$ (the reduction is $\Met_p\ni d\mapsto \frac{5d}{p}\in\Met_5$) and from $\rho_{GH}$ on $\Met^3$ to $\rho_{GH}$ on $\Met^q$ (the reduction is $\Met^3\ni d\mapsto\frac{qd}{3}\in\Met^q$). Hence, it suffices to show that there is an injective Borel-uniformly continuous reduction from $\rho_{GH}$ on $\Met_5$ to $\rho_{GH}$ on $\Met^3$.

The strategy of the proof is the following. First, we will describe a construction which to each $(M,d_M)\in\Met_{5}$ assigns $(\tilde{M},d_{\tilde{M}})\in\Met^{3}$. We will show that for every $M,N\in\Met_{5}$ we have
\[\begin{split}
	\rho_{GH}(M,N) < 1 & \implies \rho_{GH}(\tilde{M},\tilde{N}) \leq \rho_{GH}(M,N),\\
    \rho_{GH}(\tilde{M},\tilde{N}) < \tfrac{1}{6} & \implies \rho_{GH}(M,N)\leq 5 \rho_{GH}(\tilde{M},\tilde{N}).
\end{split}\]
Finally, we will show that it is possible to make such an assignment in a Borel way, that is, find an injective Borel mapping $\Met_5\ni d\mapsto \tilde{d}\in \Met^3$ such that for every $M = (\Nat,d)$, $\tilde{M}$ is isometric to $(\Nat,\tilde{d})$.

\medskip
\noindent{\emph{First step: Construction of $\tilde{M}$}}\\
Consider $(M,d_M)\in \Met_{5}$ where $M=\{m_n:n\in\Nat\}$. For any two distinct $i,j\in\Nat$ we set $I^M_{i,j}=\{k\in\Int\setsep |k|<d_M(m_i,m_j)/2\}$. Note that by our assumption on the minimal distance in $M$ we have that $|I^M_{i,j}|\geq 5$. We set
\[
	\tilde M=M\cup\big\{p^M_{i,j,k}\setsep i<j\in\Nat,k\in I^M_{i,j}\big\}.
\]
 In the sequel, when the metric space in question is clear, we shall denote the points $p^M_{i,j,k}$ just by $p_{i,j,k}$, and $I^M_{i,j}$ just by $I_{i,j}$. We define a partial distance $d'$ as follows. Fix $i<j\in\Nat$. We set $d'(m_i,m_j)=d_M(m_i,m_j)$. For any $k\in I_{i,j}$ we set $d'(m_i,p_{i,j,k})=d_M(m_i,m_j)/2+k$, $d'(m_j,p_{i,j,k})=d_M(m_i,m_j)/2-k$. Finally, for $k, k'\in I_{i,j}$ we set $d'(p_{i,j,k},p_{i,j,k'})=|k'-k|$. The function $d'$ is then extended to the whole $\tilde M$ as the greatest extension of $d'$ (graph metric), which we denote as $\hat{d}_{\tilde M}$, and finally we take its minimum with the constant 3, that is for $x,y\in \tilde M$ we set
 \[
 	d_{\tilde M}(x,y) = \begin{cases}
 		\min\{d'(x,y),3\} & \text{if } (x,y)\in\operatorname{dom}(d'),\\
        \min\{d'(x,m_i) + d'(m_i,y),3\}& \text{if there are }i,j,j',k,k'\text{ with }\\
        &j\neq j', x = p_{\min \{ i,j \}, \max \{ i,j \}, k} \text{ and }\\
        &y = p_{\min \{ i,j' \}, \max \{ i,j' \}, k'}, \\
        3 & \text{otherwise}.
 	\end{cases}
 \]
 It is easy to verify that $(\tilde{M},d_{\tilde{M}})\in\Met^{3}$.

\medskip
 \noindent\emph{Second step: for any $\varepsilon \in(0,1)$ and any $M,N\in\Met_{5}$, we have
 \[
 \rho_{GH}(M,N) < \varepsilon \Rightarrow \rho_{GH}(\tilde{M},\tilde{N}) \leq \varepsilon.
 \]} 
Fix any $\varepsilon\in(0,1)$ and $M,N\in\Met_{5}$ with $\rho_{GH}(M,N)<\varepsilon$. By Lemma \ref{lem:GHequivalenceUniformlyDiscrete}, there exists a permutation $\pi:\Nat\rightarrow\Nat$ such that for any $i\neq j\in\Nat$, $|d_M(m_i,m_j)-d_N(n_{\pi(i)},n_{\pi(j)})|<2\varepsilon$. Note that, since $\varepsilon < 1$, we have either $I^M_{i,j} = I^N_{\pi(i),\pi(j)}$ or there exists $k\in\Nat$ with either $\{-k,k\} = I^M_{i,j}\setminus I^N_{\pi(i),\pi(j)}$ or $\{-k,k\} = I^N_{\pi(i),\pi(j)}\setminus I^M_{i,j}$. Define a correspondence $\Corr$ on $\tilde M\times \tilde N$ as a union of $\Corr_1$, $\Corr_2$ and $\Corr_3$, where
  \[\begin{split}
  \Corr_1:=& \big\{(m_i,n_{\pi(i)})\setsep i\in\Nat\big\},\\
  \Corr_2:=& \bigcup_{\{i<j\setsep \pi(i)<\pi(j)\}} \bigg( \big\{(p^M_{i,j,k},p^N_{\pi(i),\pi(j),k})\setsep k\in I^M_{i,j}\cap I^N_{\pi(i),\pi(j)}\big\}\cup\\    
  & \qquad \big\{(p^M_{i,j,-k},n_{\pi(i)}), (p^M_{i,j,k},n_{\pi(j)})\setsep k\in I^M_{i,j}\setminus I^N_{\pi(i),\pi(j)}, k>0\big\}\cup\\
  & \qquad \big\{(m_i,p^N_{\pi(i),\pi(j),-k}), (m_j,p^N_{\pi(i),\pi(j),k})\setsep k\in I^N_{\pi(i),\pi(j)}\setminus I^M_{i,j}, k > 0\big\}\bigg),\\  
  \end{split}\]
  and $\Corr_3$ is defined in a similar way as $\Corr_2$: we make the union over $\{i<j\setsep \pi(j)<\pi(i)\}$ and replace $(p^M_{i,j,k},p^N_{\pi(i),\pi(j),k})$, $(m_i,p^N_{\pi(i),\pi(j),-k})$, $(m_j,p^N_{\pi(i),\pi(j),k})$ by $(p^M_{i,j,k},p^N_{\pi(j),\pi(i),-k})$, $(m_i,p^N_{\pi(j),\pi(i),k})$, $(m_j,p^N_{\pi(j),\pi(i),-k})$, respectively.  
  
  It is straightforward to check that the correspondence $\Corr$ witnesses that $\rho_{GH}(\tilde M,\tilde N)\leq \varepsilon$.
  
  \medskip
\noindent\emph{Third step: for any $\varepsilon \in(0,1/6]$ and any $M,N\in\Met_{5}$, we have}
 \[
 \rho_{GH}(\tilde{M},\tilde{N}) < \varepsilon \Rightarrow \rho_{GH}(M,N) \leq 5\varepsilon.
 \] 
Let $\varepsilon\leq 1/6$. Suppose that for some $M,N\in\Met_5$ we have $\rho_{GH}(\tilde M,\tilde N)<\varepsilon$. By Fact \ref{fact:GHbyCorrespondences} there is a correspondence $\Corr\subseteq \tilde M\times \tilde N$ such that for every $x,x'\in \tilde M$ and $y,y'\in\tilde N$ such that $(x,y)\in \Corr$ and $(x',y')\in\Corr$ we have $|d_{\tilde M}(x,x')-d_{\tilde N}(y,y')|<2\varepsilon$. Let us consider the relations $\pi$ and $\tau$ on $\Nat$ given by
\[\begin{split}
	\pi:= & \big\{(i,j)\in\Nat\times\Nat\setsep \exists y\in\tilde{N}:\; (m_i,y)\in\Corr\;\&\; d_{\tilde N} (y,n_j)<3\varepsilon\big\},\\
    \tau:= & \big\{(j,i)\in\Nat\times\Nat\setsep \exists x\in\tilde{M}:\; (x,n_j)\in\Corr\;\&\; d_{\tilde M} (m_i,x)<3\varepsilon\big\}.
  \end{split}
\]
We shall prove that $\pi$ is a bijection with inverse $\tau$ and that this bijection, due to Lemma~\ref{lem:lehciImplikace}, witnesses $\rho_{GH}(M,N) \leq 5\varepsilon$.

First, we shall prove that $\dom(\pi) = \Nat$. Fix $i\in\Nat$. If there is $n_j\in N$ with $(m_i,n_j)\in\Corr$, we have $(i,j)\in\pi$. Otherwise, pick $p^N_{h,j,k}$ with $(m_i,p^N_{h,j,k})\in\Corr$. First, assume we have $k\leq 0$. For $l > i$ pick $a_l = p^M_{i,l,k(l)}$ with $d_{\tilde M} (m_i,a_l)\in[2\varepsilon,1+2\varepsilon)$ and find $b_l\in\tilde{N}$ with $(a_l,b_l)\in\Corr$. For $l,l'>i$, $l\neq l'$, we have $d_{\tilde M}(a_l,a_{l'})\geq 4\varepsilon$; hence, $d_{\tilde N}(b_l,b_{l'})\geq 2\varepsilon > 0$ and $b_l\neq b_{l'}$. Moreover, $d_{\tilde M}(m_i,a_l)\in[2\varepsilon,1+2\varepsilon)$ implies that $d_{\tilde N}(p^N_{h,j,k},b_l)\in (0,1+4\varepsilon)$, and so by the definition of $d_{\tilde N}$, using the fact that obviously $d_{\tilde N}(n_j,p^N_{h,j,k})\geq \frac{5}{2} \geq 1 + 4\varepsilon$, we have that 
\[\begin{split}
	b_l\in\big\{n_h\} & \cup \big\{p^N_{h,j,k'}\setsep |k'-k|=1\big\}\cup \big\{p^N_{j',h,k'}\setsep j' < h, k'\in I_{j',h}\big\} \cup\\
    & \cup \big\{p^N_{h,j',k'}\setsep h<j',k'\in I_{h,j'}\big\}.
\end{split}\]

Since the first three sets are finite and we have infinitely many $b_l$'s, there are $l,l'>i$ with $b_l = p_{h,j',k'}$ and $b_{l'} = p_{h,j^{''},k^{''}}$ for some $h<j'\neq j^{''}$ and some $k', k^{''}$. Then
\[\begin{split}
	2d_{\tilde N}(p^N_{h,j,k},n_h) & = d_{\tilde N}(p^N_{h,j,k},b_l) + d_{\tilde N}(p^N_{h,j,k},b_{l'}) - d_{\tilde N}(b_l,b_{l'}) < \\
    	& < d_{\tilde M}(m_i,a_l) + d_{\tilde M}(m_i,a_{l'}) - d_{\tilde M}(a_l,a_{l'}) + 6\varepsilon = 6\varepsilon.
\end{split}\]
Therefore, we have $d_{\tilde N}(p^N_{h,j,k},n_h) < 3\varepsilon$ and $ (i,h) \in \pi $. Similarly, for $k > 0$ we get $(i,j) \in \pi$. Therefore, $i\in \dom(\pi)$ and since $i\in\Nat$ was arbitrary, $\dom(\pi) = \Nat$.

Analogously, $\dom(\tau) = \Nat$. Fix some $(i,j)\in\pi$ and $(j,k)\in \tau$. There are $y\in\tilde{N}$ and $x\in\tilde{M}$ with $\{(m_i,y),(x,n_j)\}\subseteq\Corr$ and $\max\{d_{\tilde N}(y,n_j),d_{\tilde M}(m_k,x)\}<3\varepsilon$. Hence,
\[
	d_{\tilde M} (m_i,m_k)\leq d_{\tilde M}(m_i,x) + d_{\tilde M}(x,m_k) < (d_{\tilde N}(y,n_j) + 2\varepsilon) + 3\varepsilon < 8\varepsilon
\]
which implies $d_{\tilde M}(m_i,m_k) < 3$ and $i=k$. Therefore, $\tau\circ\pi = \id$. Analogously, $\pi\circ\tau = \id$. Therefore, $\pi$ is a bijection with $\pi^{-1} = \tau$.

Let us recall that $\hat{d}_{\tilde N}$ is our notation for the greatest extension of $d'_{\tilde N}$, that is, $\hat{d}_{\tilde N}\supset d_N$ is a metric with $d_{\tilde N} = \min\{\hat{d}_{\tilde N},3\}$. In order to see that $\pi$ witnesses $d_M\simeq_{10\varepsilon} d_N$, we shall use the following claim. 

\begin{claim}
Let us have $(m_i,y)\in\Corr$ and $(m_{i'},y')\in\Corr$ for some $i\neq i'$ and $y,y'\in \tilde N$. Then
\[
	\hat{d}_{\tilde N}(y,y')\leq d_M(m_i,m_{i'}) + 4\varepsilon.
\]
\end{claim}
\begin{proof}We may suppose that $i < i'$. Pick integers $u,v$ with $d_{\tilde M}(m_i,p_{i,i',u})\in [\tfrac{3}{2},\tfrac{5}{2})$ and $d_{\tilde M}(m_{i'},p_{i,i',v})\in [\tfrac{3}{2},\tfrac{5}{2})$. Moreover, for every $ u \leq k \leq v $, pick some $ y_{k} \in \tilde{N} $ such that $ (p_{i,i',k}, y_{k}) \in \Corr $. We check first that $d_{\tilde N}(y_k,y_{k+1})\leq 1$ for $u\leq k\leq v-1$.

In order to get a contradiction, let us assume that $ d_{\tilde N}(y_k,y_{k+1}) > 1 $ for some $ u \leq k \leq v-1 $. Note that $ \hat{d}_{\tilde N}(y_k,y_{k+1}) \in (1, 1+2\varepsilon) $, since $ d_{\tilde N}(y_k,y_{k+1}) < d_{\tilde M}(p_{i,i',k}, p_{i,i',k+1}) + 2\varepsilon = 1+2\varepsilon $, and thus $ \hat{d}_{\tilde N}(y_k,y_{k+1}) = d_{\tilde N}(y_k,y_{k+1}) $ in particular. For this reason, there is $ j $ such that $ \hat{d}_{\tilde N}(y_k,n_{j}) + \hat{d}_{\tilde N}(n_{j},y_{k+1}) = \hat{d}_{\tilde N}(y_k,y_{k+1}) $. There is $ l \in \{ k, k+1 \} $ such that $ \hat{d}_{\tilde N}(y_l,n_{j}) < \frac{1}{2}(1+2\varepsilon) = \frac{1}{2} + \varepsilon $. Consequently, $ d_{\tilde M}(p_{i,i',l}, m_{\tau(j)}) < \frac{1}{2} + \varepsilon + 2\varepsilon + 3\varepsilon \leq \frac{3}{2} $. On the other hand, since $ u \leq l \leq v $, we have $ d_{\tilde M}(p_{i,i',l}, m_{p}) \geq \frac{3}{2} $ for any $ p $. This completes the verification of $d_{\tilde N}(y_k,y_{k+1})\leq 1$.

Finally, using $ d_{\tilde N}(y,y_{u}) < d_{\tilde M}(m_i,p_{i,i',u}) + 2\varepsilon < \frac{5}{2} + 2\varepsilon < 3 $ and $ d_{\tilde N}(y_{v},y') < d_{\tilde M}(p_{i,i',v},m_{i'}) + 2\varepsilon < \frac{5}{2} + 2\varepsilon < 3 $, we get
\[\begin{split}
	\hat{d}_{\tilde N}(y,y') & \leq \hat{d}_{\tilde N}(y,y_u) + \sum_{k=u}^{v-1} \hat{d}_{\tilde N}(y_k,y_{k+1}) + \hat{d}_{\tilde N}(y_v,y')\\
    & = d_{\tilde N}(y,y_u) + \sum_{k=u}^{v-1} d_{\tilde N}(y_k,y_{k+1}) + d_{\tilde N}(y_v,y')\\
    & \leq d_{\tilde M}(m_i,p_{i,i',u}) + 2\varepsilon + \sum_{k=u}^{v-1} 1 + d_{\tilde M}(p_{i,i',v},m_{i'}) + 2\varepsilon\\
    & = \hat{d}_{\tilde M}(m_i,p_{i,i',u}) + \sum_{k=u}^{v-1} \hat{d}_{\tilde M}(p_{i,i',k},p_{i,i',k+1}) + \hat{d}_{\tilde M}(p_{i,i',v},m_{i'}) + 4\varepsilon\\
    & = d_M(m_i,m_{i'}) + 4\varepsilon,
\end{split}\]
which provides the desired inequality.
\end{proof}

Now, for every $i,i'\in\Nat$, $i\neq i'$, consider $y$ and $y'$ from $\tilde{N}$ witnessing that $(i,\pi(i))\in\pi$ and $(i',\pi(i'))\in\pi$. We have
\[\begin{split}
	d_N(n_{\pi(i)},n_{\pi(i')}) & = \hat{d}_{\tilde N}(n_{\pi(i)},n_{\pi(i')}) \leq \hat{d}_{\tilde N}(n_{\pi(i)},y) + \hat{d}_{\tilde N}(y,y') + \hat{d}_{\tilde N}(y',n_{\pi(i')})\\ 
    & \leq  3\varepsilon + d_M(m_i,m_{i'}) + 4\varepsilon + 3\varepsilon.
    \end{split}
\]
Analogously, we get $d_M(m_i,m_{i'})\leq d_N(n_{\pi(i)},n_{\pi(i')}) + 10\varepsilon$; hence, $\pi$ witnesses $d_M\simeq_{10\varepsilon} d_N$ and by Lemma~\ref{lem:lehciImplikace}, $\rho_{GH}(M,N) \leq 5\varepsilon$.

\medskip
\noindent\emph{Fourth step: there is an injective Borel mapping $\Met_5\ni d\mapsto \tilde{d}\in \Met^3$ such that for every $M = (\Nat,d)$, $\tilde{M}$ is isometric to $(\Nat,\tilde{d})$.}\\
Split $\Nat$ into two disjoint infinite subsets $A$ and $B$ enumerated as $A=\{a_n\setsep n\in \Nat\}$ and $B=\{b_n\setsep n\in\Nat\}$. Moreover, let $\{(c_n,d_n)\setsep n\in\Nat\}$ be the enumeration of the set $\{(n,m)\in \Nat^2\setsep n<m\}$ given by
\[\begin{split}
\underbrace{(1,2)}_{(c_1,d_1)},\quad \underbrace{(1,3)}_{(c_2,d_2)}, \underbrace{(2,3)}_{(c_3,d_3)},\quad \ldots,\quad (1,n), (2,n), \ldots, (n-1,n),\quad \ldots
\end{split}\]
Take arbitrary $M = (\Nat,d)\in\Met_5$ enumerated as above as $\{m_i\setsep i\in \Nat\}$, where $m_i=i$ for every $i\in\Nat$. For further use we put $I^d_{i,j}:=I^M_{i,j}$ and $p^d_{i,j,k}:=p^M_{i,j,k}$ for $i<j$ and $k\in I^M_{i,j}$. Let us inductively construct bijection $\pi_d:\Nat\to \tilde{M}$. We put $\pi_d(a_n) := m_n$ for every $n\in\en$. Now consider $(c_1,d_1)$. In $\tilde M$ there are finitely many, in fact $|I_{c_1,d_1}^d|$-many, points $p_{c_1,d_1,k}$, $k\in I_{c_1,d_1}$. We enumerate them in an increasing order as $\pi_d(b_1),\ldots,\pi_d(b_{N_1})$, where $N_1=|I_{c_1,d_1}|$. Then we enumerate the points $p_{c_2,d_2,k}$, $k\in I_{c_2,d_2}$, as $\pi_d(b_{N_1+1}),\ldots,\pi_d(b_{N_1+N_2})$, where $N_2=|I_{c_2,d_2}|$, and so on. Finally, we define $\tilde{d}\in\Met^3$ as $\tilde{d}(n,k):=d_{\tilde{M}}(\pi_d(n),\pi_d(k))$, $(n,k)\in\Nat\times\Nat$. 

Now, it is not difficult to see that the map $\Met_5\ni d\mapsto \tilde{d}\in \Met^3$ is Borel and injective.
\end{proof}

\begin{thm}\label{thm:ReductionGHMettoMet_1}
There is an injective Borel-Lipschitz on small distances reduction from $\rho_{GH}$ on $\Met$ to $\rho_{GH}$ on $\Met_p$, where $p>0$ is arbitrary. Moreover, the reduction $F:\Met\to \Met_p$ is not only Borel but even continuous and for every $q>0$ and $d\in\Met^q$ we have $F(d)\in \Met^{q+p}_p$. In particular, for $q>0$ we have
\[
(\rho_{GH}\upharpoonright \Met)\; \leq_{B,u}\; (\rho_{GH}\upharpoonright \Met_1)\quad\text{ and }\quad(\rho_{GH}\upharpoonright \Met^q)\; \leq_{B,u}\; (\rho_{GH}\upharpoonright \Met^{q+p}_p).
\]
\end{thm}
\begin{proof}
Fix $p>0$. To each metric $d\in\Met$ on $\Nat$ we associate a metric $\tilde d$ on $\Nat^2$. For any two distinct points $(m,n),(m',n')\in\Nat^2$ we set
\[
\tilde d((m,n),(m',n'))=d(m,m')+p.
\]
We \emph{claim} that $\rho_{GH}(d,e) \geq \rho_{GH}(\tilde{d},\tilde{e})$, and $\rho_{GH}(d,e) \leq \rho_{GH}(\tilde{d},\tilde{e})$ whenever $\rho_{GH}(\tilde{d},\tilde{e}) < p/2$.

Indeed, fix $\varepsilon>0$ and $d,e\in\Met$. Suppose that $\rho_{GH}(d,e)<\varepsilon$. By Fact \ref{fact:GHbyCorrespondences}, there is a correspondence $\Corr\subseteq \Nat^2$ such that for $m,m',n,n'\in\Nat$ we have $|d(m,m')-e(n,n')|<2\varepsilon$ whenever $m\Corr n$ and $m'\Corr n'$. It is straightforward to see that there exists a permutation $\pi$ of $\en^2$ such that, for every $(m,n),(m',n')\in\Nat^2$, we have $m\Corr m'$ whenever $\pi(m,n)=(m',n')$. It follows that for every $(m,n)\neq (m',n')\in \Nat^2$ we have 
\[\begin{split}\big|\tilde d((m,n),(m',n'))-  & \tilde e(\pi(m,n),\pi(m',n'))\big| = 
\\ & =\big|d(m,m')-e(\pi_1(m,n),\pi_1(m',n'))\big|<2\varepsilon,
\end{split}\]
where $\pi_1(m,n)$ is the first coordinate of $\pi(m,n)$. Therefore $\tilde d\simeq_{2\varepsilon}\tilde e$, so by Lemma \ref{lem:lehciImplikace} we have $\rho_{GH}(\tilde d,\tilde e)\leq\varepsilon$. On the other hand, if  $\rho_{GH}(\tilde d,\tilde e)<\varepsilon<p/2$ then, by Lemma~\ref{lem:GHequivalenceUniformlyDiscrete}, there is a permutation $\pi$ of $\Nat^2$ witnessing $\tilde{d}\simeq_{2\varepsilon}\tilde e$ and then applying Fact~\ref{fact:GHbyCorrespondences} to the correspondence
\[
\Corr:=\big\{(i,j)\in\Nat^2\setsep \exists k,l\in\Nat\; \pi(i,k) = (j,l)\big\},
\]
we easily obtain $\rho_{GH}(d,e)\leq \varepsilon$.

Fix some bijection $\varphi:\Nat \to \Nat^2$ and define $F(d)\in\Met_p$ by $F(d)(i,j):=\tilde{d}(\varphi(i),\varphi(j))$, $(i,j)\in\Nat^2$. Then obviously $(\Nat,F(d))$ is isometric to $(\Nat^2,\tilde{d})$ and the mapping $F:\Met\to \Met_p$ is the reduction. Moreover, it is easy to see that $F$ is one-to-one and continuous.
\end{proof}

\begin{cor}\label{cor:ReductionGHMettoMet^q_p}
Fix real numbers $0<p<q$. Then there is an injective Borel-uniformly continuous reduction from $\rho_{GH}$ on $\Met$ to $\rho_{GH}$ on $\Met^q_p$.

Moreover, the reduction is not only Borel-uniformly continuous, but also Borel-Lipschitz on small distances.
\end{cor}
\begin{proof}
By Theorems \ref{thm:ReductionGHMettoMet_1} and \ref{thm:ReductionGHMet_5toMet^3} we have
\[
(\rho_{GH}\upharpoonright \Met)\; \leq_{B,u}\; (\rho_{GH}\upharpoonright \Met_1)\; \leq_{B,u}\; (\rho_{GH}\upharpoonright \Met^{q-p})\; \leq_{B,u}\; (\rho_{GH}\upharpoonright \Met^q_p).
\]
\end{proof}

\section{Continuous orbit equivalence relations}\label{section:orbitPseudometrics}
One of the main supplies of Borel and analytic equivalence relations comes from actions of Polish groups. Let $G$ be a Polish group and $X$ a Polish or standard Borel space. Suppose that $G$ acts on $X$ in a continuous or Borel way. The corresponding \emph{orbit equivalence relation} $E_G$ on $X$ is defined as $x E_G y$ if and only if $\exists g\in G\; (gx=y)$, where $x,y\in X$. The most important and most studied are the countable Borel equivalence relations. Nevertheless, there is now a rather developed theory also for actions of general (so typically not locally compact) Polish groups. See e.g. \cite{Hjo} for a reference. In particular, we highlight the fact that there is a universal orbit equivalence relation (see \cite[Theorem 5.1.9]{gao}, which can be, by the result of Gao and Kechris \cite{GaoKe}, realized as the canonical action of the isometry group of the Urysohn universal metric space $\mathbb{U}$ on the Effros-Borel space $F(\mathbb{U})$ (refer to \cite{ver} or \cite{gao} for information about the Urysohn space). Since then, several natural analytic equivalence relations have been proved to be bi-reducible with the universal orbit equivalence relations, including those that are not per se defined as orbit equivalences; see e.g. \cite{Sabok} and \cite{Zie}. 

Here we demonstrate on examples that also non-discrete pseudometrics can be defined using actions of Polish groups on standard Borel spaces equipped with an analytic pseudometric. 
\begin{defin}\label{def:generalOrbitPseudometric}
Let $G$ be a Polish group and let $X$ be a standard Borel $G$-space equipped with an analytic metric $d$ (or, more generally, an analytic pseudometric $d$), on which $G$ acts by isometries. We define an analytic pseudometric $\rho_{G,d}$ induced by the action of $G$ on $X$ as follows. For any $x,y\in X$ we set
\[
\rho_{G,d}(x,y)=\inf\{d(gx,hy)\setsep g,h\in G\}.
\]
We call such pseudometrics \emph{orbit pseudometrics}.

\end{defin}
Clearly, $\rho_{G,d}$ is an analytic pseudometric. Indeed, fix any $r>0$. Then $\{(x,y)\in X^2\setsep \rho_{G,d}(x,y)<r\}=\{(x,y)\in X^2\setsep \exists g,h\; (g,h,x,y)\in D_r\}$, where $D_r=\{(g,h,x,y)\setsep d(gx,hy)<r\}$ is analytic as $d$ is analytic and the action is Borel.

\begin{remark}
Even when $G$ does not act on $X$ by isometries, it is possible to define the orbit pseudometric by setting $$\widetilde{\rho_{G,d}}(x,y)=\inf\Big\{\sum_{i=1}^n \rho_{G,d}(z_i,z_{i+1})\setsep z_1=x,z_{n+1}=y\Big\},$$ where $\rho_{G,d}$ is defined as above.
\end{remark}

In this section we further develop the theory of orbit pseudometrics. The original motivation of the authors was to prove the following.

\begin{thm}\label{thm:ghE1}
$E_1$ is not Borel reducible to $E_{\rho_{GH}}$.
\end{thm}

It follows immediately from Corollary~\ref{cor:E1notReducibletoOurPseudoemtrics}. Actually, our proof admits an interesting generalization. We develop a new notion of a CTR orbit pseudometric, show that $E_1$ is not Borel reducible to $E_{\rho}$ whenever $\rho$ is a CTR orbit pseudometric and that there are many examples of pseudometrics $\rho'$ (including $\rho'=\rho_{GH}$) for which $E_{\rho'}$ is Borel bi-reducible to an equivalence relation given by a CTR orbit pseudometric.

\subsection{CTR orbit pseudometrics}

Note that in the full generality provided by the previous definition, every analytic pseudometric is an orbit pseudometric for trivial reasons. If $\rho$ is any analytic pseudometric on a standard Borel space, then it is an orbit pseudometric of the trivial action of any Polish group (e.g. the trivial group). This leads us to impose some natural restrictions that make the definition more interesting.

\begin{defin}\label{def:RestrictedOrbitPseudometric}
Let $G$, $X$ and $d$ be as in Definition \ref{def:generalOrbitPseudometric}. If we additionally require that $d$ is a complete metric and refines some compatible topology on $X$ (meaning that this topology defines the same Borel structure on $X$ and makes the action of $G$ continuous), then we say the orbit pseudometric $\rho_{G,d}$ is \emph{CTR} (given by a complete topology-refining metric).
\end{defin}

Our natural examples of orbit pseudometrics are CTR. It turns out that most of the natural pseudometrics, not necessarily defined as orbit pseudometrics, are Borel bi-reducible with CTR pseudometrics as well.

The main result of this section, Theorem \ref{thm:E1notreducible}, works for CTR orbit pseudometrics.\\

\noindent{\bf Examples.}

\smallskip

\begin{enumerate}[leftmargin=0cm,itemindent=.5cm,start=1,label={\bfseries \arabic*. }]

\item Let $G$ be a Polish group and $X$ a Borel $G$-space. Let $d$ be the trivial metric on $X$, i.e. $d(x,y)=0$ if $x=y$, and $d(x,y)=1$ otherwise. Then $\rho_{G,d}$ is an analytic pseudometric and $E_{\rho_{G,d}}$ is the standard orbit equivalence relation induced by the action of $G$. It was proved by Becker and Kechris in \cite[Theorem 5.2.1]{BeKe96} that there exists a Polish topology on $X$ inducing the same Borel structure so that the action of $G$ becomes continuous. This shows that $\rho_{G,d}$ is CTR.

\smallskip

\item\label{ex:gh} Consider $ X = \Met_{1/2}^{1} $, $ G = S_{\infty} $, $ (\pi \cdot f)(m, n) = f(\pi^{-1}(m), \pi^{-1}(n)) $ and $ d(f, g) = \sup_{m, n} |g(m,n)-f(m,n)| $ for $ \pi \in S_{\infty} $ and $ f, g \in \Met_{1/2}^{1} $. Then
$$ \rho_{G, d}(f, g) = \inf \{ \varepsilon : f \simeq_{\varepsilon} g \}, $$
where
\[
f \simeq_\varepsilon g \quad \Leftrightarrow \quad \exists \pi \in S_{\infty} \, \forall \{ n, m \} \in [\mathbb{N}]^{2} : |f(\pi(n), \pi(m)) - g(n, m)| \leq \varepsilon.
\]
Then $\rho_{G,d}$ is clearly a CTR orbit pseudometric.
\smallskip

\item Fix some Polish metric space $M$. Let $X=F(M)$ be the Effros-Borel space of all closed subsets of $M$. Let $G$ be $\mathrm{Iso}(M)$, the isometry group of $M$ with the pointwise convergence topology. The canonical action of $G$ on $M$ extends to a Borel action of $G$ on $X$, so $X$ is naturally a Borel $G$-space. Let $d$ be the Hausdorff metric on $X$. Clearly $d$ is Borel and the action of $G$ is by isometries with respect to $d$. We note that the resulting distance $\rho_{G,d}$ appears in \cite{Ency} where it is called \emph{generalized $G$-Hausdorff distance}.

To see that $\rho_{G,d}$ is CTR, consider the \emph{Wijsman topology} on $X$, which is the initial topology with respect to the maps $F\mapsto d_M(x,F)$, where $x\in M$. This topology is Polish (see \cite[Theorem 4.3]{Beer}) and it is compatible with the Effros-Borel structure on $X$ (see \cite[Proposition 2.6.2]{BeKe96}).
Clearly the action of $G$ on $X$ with this topology is continuous. We leave to the reader to verify the easy fact that the Hausdorff metric $d$, which is also complete, refines this topology.

\smallskip

\item\label{ex:BanachSpaceCTR} Fix some separable Banach space $E$ and let $X$ be the standard Borel space of all closed unit balls of closed linear subspaces of $E$, which we identify with the set of all linear subspaces of $E$. That is a Borel subset of $F(B_E)$. Let $G$ be $\mathrm{LIso}(E)$, the linear isometry group of $E$. The canonical action of $G$ on $E$ extends to a Borel action of $G$ on $X$. We define a Borel metric $d$ on $X$ so that $d(U,V)$, for $U,V\in X$, is the Hausdorff distance between $U$ and $V$. The action is again clearly by isometries and $\rho_{G,d}$ is an orbit pseudometric on the space $X$.

Let us check that $\rho_{G,d}$ is CTR. Consider the Wijsman topology restricted to the subspace of all closed unit balls of closed linear subspaces of $E$. This is a subset of the Polish space $F(E)$ with the Wijsman topology, therefore it is separable. We claim that it is closed. Let $F$ be a non-empty closed subset of $E$ that is not a unit ball of any closed linear subspace of $E$. It means that either
\begin{itemize}
\item $0\notin F$;
\item or $F\cancel{\subseteq} B_E$;
\item or for some $x,y\in F$, $\frac{x+y}{2}\notin F$;
\item or for some $x\in F$ and $r\in\left[-1/\|x\|,1/\|x\|\right]$, $rx\notin F$.
\end{itemize}
It is easy to check though that each of these conditions is an open condition, i.e. defines an open neighborhood of $F$ of elements satisfying the same property. Let us show it for the third condition, the rest is left for the reader. Suppose that for some $x,y\in F$, $\frac{x+y}{2}\notin F$. Since $F$ is closed, we have $\delta=\dist_{\Norm}(\frac{x+y}{2},F)>0$. Let $O$ be an open neighborhood of $F$ defined as $\{C\in F(E)\setsep \dist_{\Norm}(x,C)<\delta/2,\dist_{\Norm}(y,C)<\delta/2,\dist_{\Norm}(\frac{x+y}{2},C)>3\delta/4\}$. Pick any $C\in O$. There must exist $x',y'\in C$ such that $\|x-x'\|<\delta/2$, $\|y-y'\|<\delta/2$. However, we have $\frac{x'+y'}{2}\notin C$. Otherwise, since $\|\frac{x+y}{2}-\frac{x'+y'}{2}\|\leq \delta/2$, we would get $\dist_{\Norm}(\frac{x+y}{2},C)\leq \delta/2$, which is a contradiction. This shows that the space of unit balls of Banach subspaces of $E$ with the Wijsman topology is a Polish space. Since the Borel structure of $X$ is the restriction of the Effros-Borel structure on $F(E)$ to $X$ and the Wijsman topology on $F(E)$ is compatible with it, we get that the restriction of the Wijsman topology on $X$ is compatible with the Borel structure of $X$. The action of $G$ on $X$ is continuous, which is again easily verified. The verification that $d$ refines the Wijsman topology is done as in the previous example.

Finally, to check that $d$ restricted to $X$ is complete, it suffices to check that $X$ is closed with respect to $d$ as $d$ is complete on $F(E)$. However, $X$ is closed already with respect to the coarser Wijsman topology.

\smallskip

\item \label{ex:C*-algebras} Kadison and Kastler define in \cite{KadKas} a metric on the space of all \emph{concrete} $C^*$-algebras, i.e. sub-$C^*$-algebras of $B(\mathcal{H})$ for some Hilbert space $\mathcal{H}$. For $A,C$, subalgebras of $B(\mathcal{H})$, their Kadison-Kastler distance $d_{KK}$ is again nothing but the Hausdorff distance $\rho_H^{B(\mathcal{H})}(B_A,B_C)$. Let $\mathcal{H}$ be now a fixed separable infinite-dimensional Hilbert space. We want to define a standard Borel space of all separable $C^*$-subalgebras of $B(\mathcal{H})$. Denote by $B_1(\mathcal{H})$ the closed unit ball in $B(\mathcal{H})$, that is, the set of all operators on $\mathcal{H}$ of norm bounded by $1$. Consider the strong* operator topology ($SOT^*$) on $B_1(\mathcal{H})$. That is, a topology generated by the seminorms $B(\mathcal{H})\ni T\mapsto (\|Tx\|^2 + \|T^*x\|^2)^{1/2}$, $x\in \mathcal{H}$; that is, a net $(T_\alpha)_\alpha$ of operators converges to an operator $T$ if and only if $(T_\alpha)_\alpha$, resp. $(T^*_\alpha)_\alpha$, converge to $T$, resp. $T^*$ in the strong operator topology. It is then easy to see that $(B_1(\mathcal{H}),SOT^*)$ is a Polish space and that all the standard operations such as addition, scalar multiplication, multiplication and involution are continuous with respect to this topology. See \cite[Chapter I.3]{Blackadar} for details.

Let $X$ be the Borel subset of $F(B_1(\mathcal{H}))$, with the Effros-Borel structure inherited from $B_1(\mathcal{H})$ with the $SOT^*$-topology, consisting of all closed unit balls of separable $C^*$-subalgebras of $B(\mathcal{H})$, which we identify with the set of all separable $C^*$-subalgebras of $B(\mathcal{H})$ (note that there are different codings of separable $C^*$-subalgebras of $B(\mathcal{H})$ in the literature, see e.g. \cite{Ke98}). We define a Borel metric $d$ on $X$ so that $d(A,B) = d_{KK}(\Span A,\Span B)$ for $A,B\in X$. Now consider the action of the Polish group $U(\mathcal{H})$, the group of all unitary operators of $\mathcal{H}$ equipped with the strong operator topology (equivalent with the strong* operator topology), on 
$X$ by conjugation. That is, for $A \in X$ and $\varphi\in U(\mathcal{H})$, we have $$\varphi\cdot A = \{\varphi\circ T\circ \varphi^*\setsep T\in A\}.$$ This action defines an orbit pseudometric $\rho_{U(\mathcal{H}),d}$ on $X$.

\smallskip

We now check that $\rho_{U(\mathcal{H}),d}$ is CTR.
Fix some metric $p$ on $B_1(\mathcal{H})$ compatible with the $SOT^*$-topology. The Wijsman topology on $F(B_1(\mathcal{H}))$ induced by $p$ is again a Polish topology. We again show that $X$ is a closed subset of $F(B_1(\mathcal{H}))$ with respect to this topology. This is done similarly as in Example \ref{ex:BanachSpaceCTR}. Pick some closed subset $A\in F(B_1(\mathcal{H}))$. To verify that $A\in X$ we must check that
\begin{itemize}
\item $0\in A$;
\item if $x,y\in A$, then $\frac{x+y}{2}\in A$;
\item if $x\in A$ and $c\in\mathbb{C}$ with $|c|\in\left(0,1/\|x\|\right]$, then $cx\in A$;
\item if $x,y\in A$, then $x\cdot y\in A$;
\item if $x\in A$, then $x^*\in A$.
\end{itemize}
These are all closed conditions for the Wijsman topology, which is where we use that the operations are continuous with respect to the $SOT^*$ topology. The verifications are done similarly as in Example \ref{ex:BanachSpaceCTR}; let us show it for the last condition. Let $A\in F(B_1(\mathcal{H}))$ be such that for some $x\in A$, we have $x^*\notin A$. Then $\delta=p(x^*,A)>0$. Since the $*$-operation is continuous in the $SOT^*$-topology, there exists $\gamma>0$ such that if $p(x,z)<\gamma$, then $p(x^*,z^*)<\delta/2$, for $z\in B_1(\mathcal{H})$. Define an open neighborhood $O=\{B\in X\setsep p(x,B)<\gamma, p(x^*,B)>\delta/2\}$ of $A$ in the Wijsman topology. Clearly, $A\in O$. If $B\in O$, then there is some $z\in B$ with $p(z,x)<\gamma$. By the assumption, $p(z^*,x^*)<\delta/2$. Since $B\in O$ and therefore $p(x^*,B)>\delta/2$, we get $$p(z^*,B)\geq p(x^*,B)-p(x^*,z^*)>0,$$ so $z^*\notin B$.
It follows that the Wijsman topology on $X$ is compatible with the Borel structure of $X$. It is straightforward to check that the action of $G$ on $X$ is continuous. The completeness of $d$ will again follow as soon as we show that $d$ defines a topology which is finer than the Wijsman topology. We do it now. Pick some $x\in B_1(\mathcal{H})$, $A\in X$ and $\varepsilon$. We need to show that the set $O=\{B\in X\setsep |p(x,B)-p(x,A)|<\varepsilon\}$ is open in the topology induced by $d$. We just show that there is $\delta>0$ such that if $d(A,B)<\delta$, then $B\in O$. Here we shall without loss of generality assume that $p(y,z)$, for $y,z\in B_1(\mathcal{H})$, is equal to $$\sum_{i=1}^\infty \frac{\|y(\xi_i)-z(\xi_i)\|+\|y^*(\xi_i)-z^*(\xi_i)\|}{2^{i+1}},$$ where $(\xi_i)_i$ is some countable dense subset of unit vectors of $\mathcal{H}$. Set $\delta=\varepsilon/2$, i.e. suppose that $d(A,B)<\varepsilon/2$. We claim that $B\in O$. Otherwise there either exists $z\in B$ such that $p(x,z)\leq p(x,A)-\varepsilon$, or for all $z\in B$ we have $p(x,z)\geq p(x,A)+\varepsilon$. Suppose the former, the latter is treated analogously. Since $d(A,B)<\varepsilon/2$, there is $y\in A$ with $\|y - z\|<\varepsilon/2$. Therefore for all $i\in\Nat$, $\|y(\xi_i)-z(\xi_i)\|+\|y^*(\xi_i)-z^*(\xi_i)\|<\varepsilon$, so $p(y,z)<\varepsilon$. It follows that $$p(x,A)\leq p(x,z)+p(z,y)<p(x,A),$$ which is a contradiction.

\smallskip

\item In the theory of graph limits (see \cite{Lov} for a reference on this topic),  a \emph{graphon} is a symmetric measurable function $W:([0,1]^2,\lambda^2)\rightarrow [0,1]$, where $\lambda^2$ is the usual Lebesgue measure on $[0,1]^2$. Viewing each graphon $W$ as an element of $L^\infty([0,1]^2)$ we may equip the space $\mathcal{G}$ of all graphons with the weak*-topology coming from the identification of $L^\infty([0,1]^2)$ with $(L^1([0,1]^2))^*$, so that it becomes a compact Polish space, see \cite[Theorem F.4]{Jan}.

Equip now the linear hull of $\mathcal{G}$ with the \emph{cut norm} $\Norm_\square$ given by
\[
\|W\|_\square=\sup_{S,T} \left|\int_{S\times T} W(x,y)\, \mathrm{d} x\, \mathrm{d} y\right|,
\]
where the supremum is taken over all measurable sets $S,T\subseteq [0,1]$. Let now $G$ be the group of all measure preserving measurable bijections $\phi:[0,1]\rightarrow [0,1]$ equipped with the \emph{weak topology}, i.e. subbasic neighborhoods of a transformation $\phi$, given by a measurable set $A\subseteq [0,1]$ and $\varepsilon>0$, are of the form $\{\psi\setsep \lambda(\phi(A)\bigtriangleup \psi(A))<\varepsilon\}$. This is the strong (and also weak) topology, when the transformations are viewed as the corresponding unitary operators in $L^2([0,1])$. $G$ then becomes a Polish group, see \cite[Lemma 2.11]{Hjo} and it acts naturally on the space $\mathcal{G}$ of graphons by 
\[
gW(x,y)=W(g^{-1}x,g^{-1}y),
\]
which is obviously a continuous action.
This action together with the cut norm $\Norm_\square$ gives a pseudometric on $\mathcal{G}$, called \emph{cut distance} and denoted by $\delta_\square$ (see \cite[Section 8.2.2]{Lov}), defined as 
\[
\delta_\square(U,W)=\inf_{g,h\in G} \|gU-hW\|_\square.
\]
One can check it is a Borel pseudometric. There is a connection between $\delta_{\square}$ and ``graph limits'', since the metric quotient of $(\mathcal{G},\delta_\square)$ corresponds to the compact metrizable space of graph limits, which essentially shows that $E_{\delta_\square}$ is smooth. See \cite[Section 3.4]{BCLSV} for corresponding definitions and the result.

We claim that $\rho_{G,\Norm_\square}$ is a CTR orbit pseudometric. The cut norm $\Norm_\square$ refines the weak* topology (see \cite[Lemma 8.22]{Lov}) and we leave to the reader to verify that it is complete.

Let us note that we could work in a slightly more general setting and consider a graphon as a symmetric measurable function $W:(\Omega,\mathcal{A},\mu)^2\rightarrow [0,1]$, where $(\Omega,\mathcal{A},\mu)$ is a standard probability space, that is, a probability space defined on a standard Borel space. Even in this slightly more general setting we would end up with a Borel pseudometric $\delta_\square$.

\end{enumerate}

We conclude this subsection by showing that many standard pseudometrics are Borel bi-reducible with CTR pseudometrics.
\begin{thm}\label{thm:ghCTR}
Let $\rho$ be any analytic pseudometric from the statement of Theorem~\ref{thm:intro1}; in particular, the Gromov-Hausdorff distance $\rho_{GH}$. Then $\rho$ is Borel bi-reducible with a CTR pseudoemtric.
\end{thm}
\begin{proof}
By Lemma~\ref{lem:lehciImplikace} and Lemma~\ref{lem:GHequivalenceUniformlyDiscrete}, we have $ \rho_{GH}(f, g) = (1/2) \rho_{G, d}(f, g) $ if one of the sides is less than $ 1/4 $, where $\rho_{G,d}$ is the CTR pseudometric from Example~\ref{ex:gh}. Therefore the Gromov-Hausdorff distance when restricted to $\Met_{1/2}^1$ is Borel-uniformly continuous bi-reducible with a CTR orbit pseudometric. It follows from this observation and from Theorem~\ref{thm:intro1} that all analytic pseudometrics from the statement are Borel-uniformly continuous bi-reducible with a CTR orbit pseudometric.
\end{proof}

\subsection{Non-reducibility of the equivalence \texorpdfstring{$E_1$}{E1} into CTR orbit pseudometrics}

\begin{defin}
The equivalence relation $ E_{1} $ on $ (2^{\mathbb{N}})^{\mathbb{N}} $ is defined by
$$ x E_{1} y \quad \Leftrightarrow \quad \exists N \forall n \geq N : x(n) = y(n). $$
\end{defin}

Recall that an equivalence relation $E$ is \emph{hypersmooth} if it can be written as an increasing union $\bigcup_n E_n$ of smooth equivalence relations, i.e. equivalence relations Borel reducible to the identity relation. The relation $E_1$ is clearly hypersmooth and it plays a prominent role among hypersmooth equivalences as Kechris and Louveau prove in \cite[Theorem 2.1]{KeLo97} that every hypersmooth equivalence relation is either Borel reducible to $E_0$ on $2^\Nat$, where $xE_0 y$ if and only if $\{n\in\Nat\setsep x(n)\neq y(n)\}$ is finite, or it is Borel bi-reducible with $E_1$.

As they mention in \cite[Section 4]{KeLo97}, it is a delicate question to decide for a given analytic equivalence relation $E$ whether $E_1\leq_B E$. They show it does not happen when $E$ is Borel idealistic (see the end of this section), and when $E$ is an orbit equivalence relation. Here we extend their result for equivalences given by CTR orbit pseudometrics. Our result raises some problems that we discuss at the end of the section.

\begin{thm} \label{thm:E1notreducible}
Let $\rho_{G,d}$ be a CTR orbit pseudometric on a standard Borel space $X$. Then $ E_{1} $ is not Borel reducible to $ E_{\rho_{G, d}} $.
\end{thm}

The proof of the theorem is inspired by the proof of \cite[Theorem 4.2]{KeLo97} as presented in \cite[Section 8]{Hjo}.

By the definition of a CTR orbit pseudometric, we shall without loss of generality assume that $X$ is a Polish space and the action of $G$ is continuous. We first need the following lemma.

By $ \forall^{*} $ we mean ``for all elements of a comeager set''.

\begin{lemma}[{based on \cite[Lemma~3.17]{Hjo}}] \label{lem:comeagersetofgoodpoints}
Let $ G $ and $ H $ be Polish groups and $ X $ and $ Y $ be Polish $ G $ and $ H $-spaces. Let $ Y $ be equipped with an analytic pseudometric $ d $, on which $ H $ acts by isometries. Suppose that $ \theta : X \to Y $ is a Borel function such that
$$ \rho_{H, d}(\theta(x), \theta(g \cdot x)) = 0 $$
for all $ x \in X $ and $ g \in G $, i.e. $\theta$ is a Borel homomorphism from $E_G$ to $E_{\rho_{H,d}}$. Then, for every open neighborhood $ W $ of the identity in $ H $ and every $ \varepsilon > 0 $, there is a comeager set of $ x \in X $ for which there is an open neighborhood $ V $ of the identity in $ G $ with
$$ \forall^{*} g \in V \exists h \in W : d(\theta(g \cdot x), h \cdot \theta(x)) < \varepsilon. $$
\end{lemma}

Let us fix a neighborhood $W$ and $\varepsilon>0$. To prove the lemma, we need the following claim which is an analogue of a claim from the proof of \cite[Lemma~3.17]{Hjo}. In fact, the argument demonstrating that the lemma follows from the claim is the same as in \cite{Hjo}, and so we omit it.

\begin{claim}
For all $ x \in X $, there is a comeager set of $ g_{0} \in G $ for which there exists some open neighborhood $ V $ of the identity in $ G $ with
$$ \forall^{*} g_{1} \in V \exists h \in W : d(\theta(g_{1}g_{0} \cdot x), h \cdot \theta(g_{0} \cdot x)) < \varepsilon. $$ 
\end{claim}

\begin{proof}
Fix $x\in X$ and choose a smaller open neighborhood $ W' $ of $ 1_{H} $ with $ (W')^{-1} = W' $ and $ (W')^{2} \subseteq W $. 
Taking a sequence $ (h_{i})_{i=1}^{\infty} $ in $ H $ such that $ \{ W' \cdot h_{i} : i \in \mathbb{N} \} $ covers $ H $, we obtain a cover of $ G $ by the sets $$C'_i=\big\{g\in G\setsep \exists h\in W'\;(d(\theta(g \cdot x), h h_{i} \cdot \theta(x)) < \varepsilon/2)\big\},\quad i \in\Nat.$$ Indeed, given $ g \in G $, we have $ \rho_{H, d}(\theta(g \cdot x), \theta(x)) = 0 $, and so there is $ h' \in H $ such that $ d(\theta(g \cdot x), h' \cdot \theta(x)) < \varepsilon/2 $. For some $ i \in \mathbb{N} $ and $ h \in W' $, we have $ h' = h h_{i} $.

Since each $C'_i$ is analytic and therefore has the Baire property, there is an open set $O_i$ such that the symmetric difference of $C'_i$ and $O_i$ is meager. Set $C_i=C'_i\cap O_i$. Clearly, $C=\bigcup_{i\in\Nat} C_i$ is comeager. Take any $g_0\in C$. There are $i\in\Nat$ with $g_0\in C_i$. We put $V = O_ig_0^{-1}$. Then $\forall^* g_1\in V$ we have $g_1g_0\in C_i$, so it suffices to check
$$ g_{0} \in C_{i} \, \& \, g_{1} g_{0} \in C_{i} \quad \Rightarrow \quad \exists h \in W : d(\theta(g_{1} g_{0} \cdot x), h \cdot \theta(g_{0} \cdot x)) < \varepsilon $$
for $ g_{0}, g_{1} \in G $. There are $ h', h'' \in W' $ such that $ d(\theta(g_{0} \cdot x), h' h_{i} \cdot \theta(x)) < \varepsilon/2 $ and $ d(\theta(g_{1} g_{0} \cdot x), h'' h_{i} \cdot \theta(x)) < \varepsilon/2 $. Since
\begin{align*}
d(\theta( & g_{1} g_{0} \cdot x), h'' (h')^{-1} \cdot \theta(g_{0} \cdot x)) \\
 & \leq d(\theta(g_{1} g_{0} \cdot x), h'' h_{i} \cdot \theta(x)) + d(h'' (h')^{-1} \cdot \theta(g_{0} \cdot x), h'' h_{i} \cdot \theta(x)) \\
 & = d(\theta(g_{1} g_{0} \cdot x), h'' h_{i} \cdot \theta(x)) + d(\theta(g_{0} \cdot x), h' h_{i} \cdot \theta(x)) < \varepsilon/2 + \varepsilon/2 = \varepsilon,
\end{align*}
the choice $ h = h'' (h')^{-1} $ works.
\end{proof}

As mentioned above, the proof of Lemma~\ref{lem:comeagersetofgoodpoints} is finished as well.

\begin{proof}[Proof of Theorem \ref{thm:E1notreducible}]
In order to get a contradiction, let us assume that $ \theta : (2^{\mathbb{N}})^{\mathbb{N}} \to X $ is a Borel map with
$$ x E_{1} y \Leftrightarrow \rho_{G, d}(\theta(x), \theta(y)) = 0, \quad x, y \in (2^{\mathbb{N}})^{\mathbb{N}}. $$
Due to \cite[Lemma~11.8.2]{kanovei}, we may assume that $ \theta $ is continuous. We notice first that, by Lemma~\ref{lem:comeagersetofgoodpoints}, the subset of $ (2^{\mathbb{N}})^{\mathbb{N}} $ given by
\begin{align*}
C_{n} = \big\{ x : \; & \forall W \textrm{ nbhd of } 1_{G}, \forall \varepsilon > 0, \exists V \textrm{ nbhd of } (0, 0, \dots), \forall^{*} a \in V, \exists g \in W \\
 & d\big(\theta(x(1), \dots, x(n-1), x(n) +_{2} a, x(n+1), \dots), g \cdot \theta(x)\big) < \varepsilon \big\}
\end{align*}
is comeager for every $ n \in \mathbb{N} $. More precisely, we apply the lemma on every $ W $ from a countable neighborhood basis of $ 1_{G} $ and on every $ \varepsilon $ of the form $ 1/k $, $ G $ and $ X $ from the lemma are $ 2^{\mathbb{N}} $ and $ (2^{\mathbb{N}})^{\mathbb{N}} $ with the action $ a \cdot x = (x(1), \dots, x(n-1), x(n) +_{2} a, x(n+1), \dots) $, and $ H $ and $ Y $ from the lemma are the current $ G $ and $ X $.

Using the Kuratowski-Ulam theorem (see e.g. \cite[Theorem~2.46]{Hjo}), we can pick $ x \in (2^{\mathbb{N}})^{\mathbb{N}} $ such that
\begin{align*}
& x = (x(1), x(2), x(3), \dots) \in C_{1}, \\
& \forall^{*} a_{1} \in 2^{\mathbb{N}} : (a_{1}, x(2), x(3), \dots) \in C_{2}, \\
& \forall^{*} a_{1} \in 2^{\mathbb{N}} \, \forall^{*} a_{2} \in 2^{\mathbb{N}} : (a_{1}, a_{2}, x(3), \dots) \in C_{3}, \\
& \quad \vdots
\end{align*}
Indeed, for every $n\in\Nat$ the set $$D'_n=\big\{(x(n),x(n+1),\ldots)\setsep C_n^{(x(n),\ldots)}\text{ is comeager}\big\}$$ is comeager, where $C_n^{(x(n),\ldots)}=\{(x(1),\ldots,x(n-1))\setsep (x(1),\ldots,x(n),\ldots)\in C_n\}$. Set $$D_n=\big\{x\in (2^{\mathbb{N}})^{\mathbb{N}}\setsep (x(1),\ldots,x(n-1))\in (2^{\mathbb{N}})^{n-1}, (x(n),x(n+1),\ldots)\in D'_n\big\}.$$ Then each $D_n$ is comeager in $(2^{\mathbb{N}})^{\mathbb{N}}$ and we can take $x\in \bigcap_n D_n$.

Let $ p $ be a compatible complete metric on $ G $ and let $ g_{0} = 1_{G} $. We show that it is possible to find $ y = (y(1), y(2), y(3), \dots) \in (2^{\mathbb{N}})^{\mathbb{N}} $ and $ g_{1}, g_{2}, \dots \in G $ such that, denoting
$$ x_{n} = (y(1), \dots, y(n), x(n+1), \dots), \quad n = 0, 1, \dots, $$
we have for every $ n \in \mathbb{N} $ that
\begin{itemize}
\item[{(i)}] $ y(n) \neq x(n) $,
\item[{(ii)}] $ p(g_{n}, g_{n-1}) < 1/2^{n} $,
\item[{(iii)}] $ d(g_{n} \cdot \theta(x_{n}), g_{n-1} \cdot \theta(x_{n-1})) < 1/2^{n} $,
\item[{(iv)}] it holds that
\begin{align*}
& x_{n} = (y(1), \dots, y(n), x(n+1), x(n+2), x(n+3), \dots) \in C_{n+1}, \\
& \forall^{*} a_{n+1} \in 2^{\mathbb{N}} : (y(1), \dots, y(n), a_{n+1}, x(n+2), x(n+3), \dots) \in C_{n+2}, \\
& \forall^{*} a_{n+1} \in 2^{\mathbb{N}} \, \forall^{*} a_{n+2} \in 2^{\mathbb{N}} : (y(1), \dots, y(n), a_{n+1}, a_{n+2}, x(n+3), \dots) \in C_{n+3}, \\
& \quad \vdots
\end{align*}
\end{itemize}
Let us note first that (iv) already holds for $ n = 0 $ due to the choice of $ x_{0} = x $. Assume that $ n \in \mathbb{N} $ and that $ y(1), \dots, y(n-1) $ and $ g_{1}, \dots, g_{n-1} $ are already found. Let $ W = \{ g \in G : p(g_{n-1} g^{-1}, g_{n-1}) < 1/2^{n} \} $. Since $ x_{n-1} \in C_{n} $, considering $ \varepsilon = 1/2^{n} $ in the definition of $ C_{n} $, we have
\begin{align*}
 & \exists V \textrm{ nbhd of } (0, 0, \dots), \forall^{*} a \in V, \exists g \in W : \\
 & \quad \quad d(\theta(y(1), \dots, y(n-1), x(n) +_{2} a, x(n+1), \dots), g \cdot \theta(x_{n-1})) < 1/2^{n}.
\end{align*}
Let us take such open neighborhood $ V $. We obtain from condition (iv) for $ n - 1 $ that
\begin{align*}
& \forall^{*} a_{n} \in x(n) +_{2} V : (y(1), \dots, y(n-1), a_{n}, x(n+1), x(n+2), \dots) \in C_{n+1}, \\
& \forall^{*} a_{n} \in x(n) +_{2} V, \; \forall^{*} a_{n+1} \in 2^{\mathbb{N}} : \; (y(1), \dots, a_{n}, a_{n+1}, x(n+2), \dots) \in C_{n+2}, \\
& \quad \vdots
\end{align*}
Hence, we can choose $ y(n) \in x(n) +_{2} V $ such that $ y(n) \neq x(n) $,
$$ \exists g \in W : \, d(\theta(y(1), \dots, y(n-1), y(n), x(n+1), \dots), g \cdot \theta(x_{n-1})) < 1/2^{n} $$
and
\begin{align*}
& (y(1), \dots, y(n-1), y(n), x(n+1), x(n+2), \dots) \in C_{n+1}, \\
& \forall^{*} a_{n+1} \in 2^{\mathbb{N}} : (y(1), \dots, y(n-1), y(n), a_{n+1}, x(n+2), \dots) \in C_{n+2}, \\
& \quad \vdots
\end{align*}
Provided with $ g \in W $ such that $ d(\theta(x_{n}), g \cdot \theta(x_{n-1})) < 1/2^{n} $, we define $ g_{n} = g_{n-1} g^{-1} $. Then $ p(g_{n}, g_{n-1}) < 1/2^{n} $ (due to the choice of $ W $) and $ d(g_{n} \cdot \theta(x_{n}), g_{n-1} \cdot \theta(x_{n-1})) = d(\theta(x_{n}), g \cdot \theta(x_{n-1})) < 1/2^{n} $. Therefore, our choice of $ y(n) $ and $ g_{n} $ works, as conditions (i)--(iv) hold for $ n $.

So, we have seen that there are appropriate $ y = (y(1), y(2), y(3), \dots) $ and $ g_{1}, g_{2}, \dots $ indeed. It is clear from (i) that $ (x, y) \notin E_{1} $. To obtain the promised contradiction, we provide a series of simple arguments resulting to $ (x, y) \in E_{1} $.

Notice that $ x_{n} \to y $. By (ii), the sequence $ g_{1}, g_{2}, \dots $ has a limit $ g $ in $ G $. Considering the continuity of $ \theta $, we obtain $ g_{n} \cdot \theta(x_{n}) \to g \cdot \theta(y) $ in the original topology of $ X $. By (iii), the sequence $ g_{n} \cdot \theta(x_{n}) $ is Cauchy in $ (X, d) $. Since $d$ is a complete and the topology refining metric, this sequence has a limit in $ (X, d) $ which is nothing but the same point $ g \cdot \theta(y) $. Using (iii) again, we arrive at
$$ d(g \cdot \theta(y), g_{n} \cdot \theta(x_{n})) < 1/2^{n+1} + 1/2^{n+2} + \dots = 1/2^{n}, $$
and so
$$ \rho_{G, d}(\theta(y), \theta(x_{n})) < 1/2^{n}, $$
for every $ n \in \mathbb{N} $. Since $ x E_{1} x_{n} $, we have $ \rho_{G, d}(\theta(x), \theta(x_{n})) = 0 $. Thus,
$$ \rho_{G, d}(\theta(x), \theta(y)) \leq \rho_{G, d}(\theta(x), \theta(x_{n})) + \rho_{G, d}(\theta(y), \theta(x_{n})) < 0 + 1/2^{n} $$
for every $ n \in \mathbb{N} $. For this reason, $ \rho_{G, d}(\theta(x), \theta(y)) = 0 $ and, consequently, $ x E_{1} y $.
\end{proof}

\begin{cor}\label{cor:E1notReducibletoOurPseudoemtrics}
The relation $ E_{1} $ is not Borel reducible to $E_\rho$, where $\rho$ is any of the pseudometrics from Theorem \ref{thm:intro1} \eqref{thm:intro1:(1)}or $\rho$ is any of the CTR pseudometrics mentioned in the examples above.
\end{cor}

\begin{proof}
Theorem~\ref{thm:E1notreducible} can be directly applied to the CTR orbit pseudometrics from the examples above and, by Theorem~\ref{thm:ghCTR}, also to the other distances from the statement.
\end{proof}

The corollary has an important consequence that ought to be investigated further. It has been conjectured (see e.g. the introduction in \cite{kanovei}, or \cite[Conjecture 1]{HjoKe97} and \cite[Question 10.9]{Hjo} where it was stated for Borel equivalence relations) that the equivalence relation $E_1$ is the least equivalence which is not Borel reducible to an orbit equivalence relation. This combined with Corollary \ref{cor:E1notReducibletoOurPseudoemtrics} suggests two different scenarios:
\begin{itemize}
\item If the conjecture is true (for the analytic equivalence relations), then the equivalence relations given by pseudometrics from Theorem \ref{thm:intro1} \eqref{thm:intro1:(1)}
are reducible to an orbit equivalence relation. That would actually imply that they are bi-reducible with the universal orbit equivalence relation, see the discussion in Section \ref{section:problems}.
\item If one feels that the equivalence relations given by pseudometrics from Theorem \ref{thm:intro1} \eqref{thm:intro1:(1)} should be different from orbit equivalence relations, then he or she is led to the reconsideration of the conjecture.
\end{itemize}
\begin{conjecture}
The equivalence $E_1$ is \emph{not} the least analytic equivalence relation which is not Borel reducible to an orbit equivalence.
\end{conjecture}

We note that the conjecture that $E_1$ is the least equivalence which is not Borel reducible to an orbit equivalence relation has been verified affirmatively by Solecki for the special class of equivalence relations $E_\mathcal{I}$ on $2^\Nat$, where $\mathcal{I}$ is an ideal of subsets of $\Nat$. We then have $x E_\mathcal{I} y$ if and only if $\{n\in\Nat\setsep x(n)\neq y(n)\}\in\mathcal{I}$. Note that one can view each such ideal $\mathcal{I}$ as a subgroup of $2^\Nat$. Call $\mathcal{I}$ \emph{polishable} if there exists a Polish group topology on $\mathcal{I}$ producing the same Borel structure as the standard topology on $\mathcal{I}$ inherited from $2^\Nat$.  The following follows from \cite[Theorem 2.1]{Sol} and \cite[Corollary 11.8.3]{kanovei}.
\begin{thm}[Solecki]
Let $\mathcal{I}$ be an analytic ideal on $\Nat$. Then $E_1\leq_B E_\mathcal{I}$ if and only if $\mathcal{I}$ is not polishable.
\end{thm}

It is mentioned in \cite[Chapter 8]{gao} that a plausible conjecture is that orbit equivalence relations coincide with the idealistic equivalence relations, which is motivated by the fact that Kechris and Louveau prove in \cite{KeLo97} that $E_1$ is not Borel reducible to any Borel idealistic equivalence relation. Note that Kechris and Louveau in \cite{KeLo97} pose the problem whether Borel idealistic equivalence relations coincide with Borel equivelence relations $E$ such that $E_1\cancel{\leq}_B E$.

Recall that an equivalence relation $E$ on a standard Borel space $X$ is \emph{idealistic} if for every equivalence class $C\subseteq X$ of $E$ there is a $\sigma$-ideal $\mathcal{I}_C$ of subsets of $C$ such that
\begin{itemize}
\item $C\notin \mathcal{I}_C$;
\item for every Borel set $A\subseteq X^2$ the set $\{x\in X\setsep \{y\in [x]_E\setsep (x,y)\in A\}\in \mathcal{I}_{[x]_E}\}$ is Borel.
\end{itemize}
In view of the conjecture, the following is natural to be investigated.
\begin{question}
Are the equivalences $E_\rho$, where $\rho$ is a CTR orbit pseudometric, idealistic?
\end{question}

\section{Distances are not Borel}\label{section:notBoreldist}
Although the pseudometrics are in general analytic, the natural examples mentioned in Theorem~\ref{thm:intro1} and thoroughly investigated in the complementary paper of the authors (\cite{CDKpart2}) have a special form. They are analytic non-Borel, however by combination of results from \cite{CDKpart2} and from \cite{BYDNT} (or see Appendix~\ref{sectionGames} for alternative and more general proofs), the equivalence classes of all these pseudometrics, except the uniform distance for which we do not know the answer, are Borel. It is then of interest to investigate whether for any of those pseudometrics $\rho$ and an element $A$ from the domain of $\rho$ we have that the sets $\{B\setsep \rho(A,B)\leq r\}$ are Borel, where $r>0$. This question was raised for the Gromov-Hausdorff and Kadets distances in \cite[Question 8.4 and p. 28]{BYDNT}. The goal of this section is to provide a negative answer and to show that whenever $\rho$ is a pseudometric to which the Kadets distance is reducible, then there are $\rho$-balls which are not Borel. The following is thus the main result of this section answering the question from \cite{BYDNT} for the Kadets distance, and using Theorem~\ref{thm:intro1} it also provides the answer for the Gromov-Hausdorff distance.
\begin{thm}\label{thm:distancenotBorel}
Let $ \rho $ be an analytic pseudometric on a standard Borel space $ P $. If $ \rho_{K} $ on $ \Banach $ is Borel-uniformly continuous reducible to $ \rho $ on $ P $, then the pseudometric $ \rho $ is not Borel. In fact, there is $ x \in P $ such that the function $ \rho(x, \cdot) $ is not Borel.
\end{thm}

The rest of the section is devoted to the proof.
Let us introduce some notation and definitions first. By $ \mathcal{P}(\mathbb{N}) $ we denote the set of all subsets of $ \mathbb{N} $ endowed with the coarsest topology for which $ \{ A \in \mathcal{P}(\mathbb{N}) : n \in A \} $ is clopen for every $ n $. Obviously, $ \mathcal{P}(\mathbb{N}) $ is nothing else than a copy of the Cantor space $ 2^{\mathbb{N}} $. Further, by $ K(\mathcal{P}(\mathbb{N})) $ we mean the hyperspace of all compact subsets of $ \mathcal{P}(\mathbb{N}) $ endowed by the Vietoris topology.

For $ E \subseteq \mathbb{N} $ and $ x \in c_{00} $, we denote by $ Ex $ the element of $ c_{00} $ given by $ Ex(n) = x(n) $ for $ n \in E $ and $ Ex(n) = 0 $ for $ n \notin E $.

If $ X $ and $ Y $ are Banach spaces, then by $ X \oplus_{1} Y $ we mean the direct sum $ X \oplus Y $ with the norm $ \Vert (x, y) \Vert = \Vert x \Vert + \Vert y \Vert $. If $ X_{1}, X_{2}, \dots $ is a sequence of Banach spaces, then its $ \ell_{1} $-sum $ (\bigoplus X_{n})_{\ell_{1}} $ is defined as the space of all sequences $ x = (x_{1}, x_{2}, \dots), x_{k} \in X_{k}, $ such that $ \Vert x \Vert := \sum_{k=1}^{\infty} \Vert x_{k} \Vert < \infty $.

We say that a sequence $ G_{1}, G_{2}, \dots $ of finite-dimensional Banach spaces is \emph{dense} if for any finite-dimensional Banach space $ G $ and any $ \varepsilon > 0 $, there is $ n \in \mathbb{N} $ such that $ \mathrm{dim} \, G_{n} = \mathrm{dim} \, G $ and $ \rho_{BM}(G_{n}, G) < \varepsilon $.

In the context of Banach spaces, by a basis we mean a Schauder basis. By a \emph{basic sequence} we mean a basis of its closed linear span. A basis $ \{ x_{i} \}_{i=1}^{\infty} $ of a Banach space $ X $ is said to be \emph{shrinking} if
$$ X^{*} = \overline{\mathrm{span}} \{ x_{1}^{*}, x_{2}^{*}, \dots \}, $$
where $ x_{1}^{*}, x_{2}^{*}, \dots $ is the dual basic sequence $ x_{n}^{*} : \sum_{i=1}^{\infty} a_{i}x_{i} \mapsto a_{n} $.

We say that a sequence $ x_{1}, x_{2}, \dots $ of non-zero vectors in a Banach space $ X $ is \emph{$ c $-equivalent to the standard basis of $ \ell_{1} $} if $ \Vert \sum_{k=1}^{n} \alpha_{k} x_{k} \Vert \geq \frac{1}{c} \sum_{k=1}^{n} \Vert \alpha_{k} x_{k} \Vert $ for all $ n \in \mathbb{N} $ and $ \alpha_{1}, \dots, \alpha_{n} \in \mathbb{R} $.

\begin{prop} \label{propnonBoreldist}
Let us consider the space
$$ X = \Big( \bigoplus G_{n} \Big)_{\ell_{1}}, $$
where $ G_{1}, G_{2}, \dots $ is a dense sequence of finite-dimensional spaces. Then, for every $ \varepsilon > 0 $, there exists a Borel mapping $ \mathfrak{S} : K(\mathcal{P}(\mathbb{N})) \to \Banach $ such that

{\rm (a)} if $ \mathcal{A} \in K(\mathcal{P}(\mathbb{N})) $ contains an infinite set, then $ \rho_{BM}(\mathfrak{S}(\mathcal{A}), X) \leq \varepsilon $, and thus $ \rho_{K}(\mathfrak{S}(\mathcal{A}), X) \leq \varepsilon $,

{\rm (b)} if $ \mathcal{A} \in K(\mathcal{P}(\mathbb{N})) $ consists of finite sets only, then $ \mathfrak{S}(\mathcal{A}) $ contains a normalized $ 1 $-separated shrinking basic sequence, and thus $ \rho_{K}(\mathfrak{S}(\mathcal{A}), X) \geq 1/8 $.
\end{prop}

\begin{proof}
To find an appropriate mapping $ \mathfrak{S} $, we apply a construction provided in \cite[\S 1(a)]{AD} and a simple idea from \cite[Remark 3.10(vii)]{ku}. Let us recall first that, for $ \mathcal{A} \in K(\mathcal{P}(\mathbb{N})) $ and $ 0 < \theta < 1 $, a Tsirelson type space $ T[\mathcal{A}, \theta] $ is defined as the completion of $ c_{00} $ under the implicitly defined norm
$$ \Vert x \Vert_{\mathcal{A}, \theta} = \max \bigg\{ \Vert x \Vert_{\infty}, \theta \sup \sum_{k=1}^{n} \Vert E_{k}x \Vert_{\mathcal{A}, \theta} \bigg\}, $$
where the \textquotedblleft sup\textquotedblright {} is taken over all finite families $ \{ E_{1}, \dots, E_{n} \} $ of finite subsets of $ \mathbb{N} $ such that
$$ \exists A \in \mathcal{A} \; \exists m_{1}, \dots, m_{n} \in A : \; m_{1} \leq E_{1} < m_{2} \leq E_{2} < \dots < m_{n} \leq E_{n}. $$
Given $ \varepsilon > 0 $, we put $ \theta = e^{-\varepsilon} $ and
$$ X_{\mathcal{A}} = T[\mathcal{A}_{1}, \theta] \oplus_{1} X, \quad \mathcal{A} \in K(\mathcal{P}(\mathbb{N})), $$
where $ \mathcal{A}_{1} = \{ A \cup \{ 1 \} : A \in \mathcal{A} \} $. We check that $ X_{\mathcal{A}} $ satisfies the requirements (a) and (b) on $ \mathfrak{S}(\mathcal{A}) $.

(a) We observe that there is a sequence of finite-dimensional spaces whose $ \ell_{1} $-sum has the Banach-Mazur distance to $ T[\mathcal{A}_{1}, \theta] $ at most $ \varepsilon $. Since $ \mathcal{A} $ contains an infinite set, we can put $ m_{1} = 1 $ and find numbers $ 1 < m_{2} < m_{3} < \dots $ such that $ \{ m_{1}, m_{2}, m_{3}, \dots \} \in \mathcal{A}_{1} $. Considering $ E_{k} = \{ m_{k}, \dots, m_{k+1} - 1 \} $ for every $ k \in \mathbb{N} $, we obtain
$$ e^{-\varepsilon} \sum_{k=1}^{\infty} \Vert E_{k}x \Vert_{\mathcal{A}_{1}, \theta} \leq \Vert x \Vert_{\mathcal{A}_{1}, \theta} \leq \sum_{k=1}^{\infty} \Vert E_{k}x \Vert_{\mathcal{A}_{1}, \theta}, \quad x \in c_{00}, $$
(the first inequality holds due to the definition of $ \Vert \cdot \Vert_{\mathcal{A}_{1}, \theta} $ and the choice $ \theta = e^{-\varepsilon} $, the second one is just the triangle inequality). So, the sequence $ \mathrm{span} \{ e_{n} : n \in E_{k} \}, k = 1, 2, \dots, $ works.

It follows that the same holds for $ X_{\mathcal{A}} $ and that an appropriate sequence of finite-dimensional spaces can be chosen to be dense. Indeed, we can collect all spaces $ \mathrm{span} \{ e_{n} : n \in E_{k} \} $ with all $ G_{n} $'s. Thus, to show that $ \rho_{BM}(X, X_{\mathcal{A}}) \leq \varepsilon $, it is sufficient to realize that $ \rho_{BM}(X, Y) = 0 $ for
$$ Y = \Big( \bigoplus H_{n} \Big)_{\ell_{1}}, $$
where $ H_{1}, H_{2}, \dots $ is another dense sequence of finite-dimensional spaces.

Given $ \delta > 0 $, we obtain by a back-and-forth argument that there are a bijection $ \pi : \mathbb{N} \to \mathbb{N} $ and linear surjective mappings $ L_{n} : G_{n} \to H_{\pi(n)} $ such that $ \Vert L_{n} \Vert \leq 1 $ and $ \Vert L_{n}^{-1} \Vert \leq e^{\delta} $. Then the operator
$$ L : X \to Y, \quad \sum_{n=1}^{\infty} x_{n} \mapsto \sum_{n=1}^{\infty} L_{n} x_{n}, \quad (x_{n} \in G_{n}), $$
satisfies $ \Vert L \Vert \leq 1 $ and its inverse
$$ L^{-1} : Y \to X, \quad \sum_{n=1}^{\infty} y_{n} \mapsto \sum_{n=1}^{\infty} L_{n}^{-1} y_{n}, \quad (y_{n} \in H_{\pi(n)}), $$
satisfies $ \Vert L^{-1} \Vert \leq e^{\delta} $. Hence, $ \rho_{BM}(X, Y) \leq \delta $. As $ \delta > 0 $ could be arbitrary, we arrive at $ \rho_{BM}(X, Y) = 0 $.

Finally, using \cite[Proposition~6.2]{o94} (or \cite[Proposition~2.1]{DuKa}), we arrive at $ \rho_{K}(X_{\mathcal{A}}, X) \leq \rho_{BM}(X_{\mathcal{A}}, X) \leq \varepsilon $.

(b) Since $ \mathcal{A} $ consists of finite sets, the canonical basis $ e_{n} = \mathbf{1}_{\{ n \}} $ of $ c_{00} $ is a shrinking basis of $ T[\mathcal{A}_{1}, \theta] $ (to show this, it is possible to adapt the part (a) of the proof of \cite[Proposition~1.1]{AD} if we consider $ \theta_{k} = \theta $ and $ \mathcal{M}_{k} = \mathcal{A}_{1} $ for every $ k $). As $ T[\mathcal{A}_{1}, \theta] \subseteq X_{\mathcal{A}} $, we obtain that $ e_{1}, e_{2}, \dots $ is a shrinking basic sequence as desired. In order to get a contradiction, let us assume that $ \rho_{K}(X, Y) < 1/8 $ where $ Y = X_{\mathcal{A}} $.

Let us choose some $ \eta $ with $ \rho_{K}(X, Y) < \eta < 1/8 $. Let $ \iota_{X} $ and $ \iota_{Y} $ be linear isometric embeddings of $ X $ and $ Y $ into a Banach space $ Z $ such that $ \rho_{H}^{Z}(\iota_{X}(B_{X}), \iota_{Y}(B_{Y})) < \eta $. Let $ x'_{1}, x'_{2}, \dots $ be points in $ B_{X} $ such that
$$ \Vert \iota_{X}(x'_{n}) - \iota_{Y}(e_{n}) \Vert_{Z} < \eta, \quad n \in \mathbb{N}. $$
It is straightforward to check that the sequence $ x_{n} = x'_{n}/\Vert x'_{n} \Vert_{X} $ fulfills
$$ \Vert \iota_{X}(x_{n}) - \iota_{Y}(e_{n}) \Vert_{Z} < 2\eta, \quad n \in \mathbb{N}. $$
This sequence is $ (1-4\eta) $-separated, as $ \Vert x_{n} - x_{m} \Vert_{X} > \Vert e_{n} - e_{m} \Vert_{\mathcal{A}_{1}, \theta} - 4\eta \geq 1 - 4\eta $ for $ n \neq m $.

We employ the fact that the space $ X $ has the $ 1 $-strong Schur property (by \cite[p. 57]{gkl}, a space is said to have the \emph{$ 1 $-strong Schur property} if, for any $ \delta \in (0, 2] $, any $ c > 2/\delta $ and any normalized $ \delta $-separated sequence, there is a subsequence which is $ c $-equivalent to the standard basis of $ \ell_{1} $). This fact follows for instance from \cite[Proposition~4.1]{KalSpu} and the observation that the proof of \cite[Theorem~1.3]{KalSpu} works for $ X $.

We obtain that $ x_{1}, x_{2}, \dots $ has a subsequence $ x_{n_{k}} $ which is $ 4 $-equivalent to the standard basis of $ \ell_{1} $. It follows that
$$ \Big\Vert \sum_{k=1}^{l} \lambda_{k} e_{n_{k}} \Big\Vert_{\mathcal{A}_{1}, \theta} \geq \Big\Vert \sum_{k=1}^{l} \lambda_{k} x_{n_{k}} \Big\Vert_{X} - \sum_{k=1}^{l} |\lambda_{k}| \cdot 2\eta \geq \Big( \frac{1}{4} - 2\eta \Big) \sum_{k=1}^{l} |\lambda_{k}| $$
for every $ l \in \mathbb{N} $ and $ \lambda_{1}, \dots, \lambda_{l} \in \mathbb{R} $. Therefore, $ e_{1}, e_{2}, \dots $ has a subsequence equivalent to the standard basis of $ \ell_{1} $. This is a contradiction, as $ e_{1}, e_{2}, \dots $ is a shrinking basic sequence at the same time.

So, (a) and (b) are proven for $ X_{\mathcal{A}} $, and it remains to show that there is a Borel mapping $ \mathfrak{S} : K(\mathcal{P}(\mathbb{N})) \to \Banach $ such that $ \mathfrak{S}(\mathcal{A}) $ is isometric to $ X_{\mathcal{A}} $ for every $ \mathcal{A} \in K(\mathcal{P}(\mathbb{N})) $. Let $ x_{1}, x_{2}, \dots $ be a sequence of linearly independent vectors in $ X $ whose linear span is dense in $ X $. For every $ \mathcal{A} \in K(\mathcal{P}(\mathbb{N})) $, we define the norm on $ V $ given by
$$ \big\Vert (q_{j})_{j=1}^{\infty} \big\Vert = \Big\Vert \sum_{k=1}^{\infty} q_{2k-1} e_{k} \Big\Vert_{\mathcal{A}_{1}, \theta} + \Big\Vert \sum_{k=1}^{\infty} q_{2k} x_{k} \Big\Vert_{X}. $$
In this way, the coded space is isometric to $ X_{\mathcal{A}} $.

Thus, we need just to check that the defined mapping is Borel, i.e., that the function
$$ \mathcal{A} \in K(\mathcal{P}(\mathbb{N})) \quad \mapsto \quad \Big\Vert \sum_{k=1}^{\infty} q_{2k-1} e_{k} \Big\Vert_{\mathcal{A}_{1}, \theta} + \Big\Vert \sum_{k=1}^{\infty} q_{2k} x_{k} \Big\Vert_{X} $$
is Borel for a fixed $ (q_{j})_{j=1}^{\infty} \in V $. If we pick $ l \in \mathbb{N} $ such that $ q_{j} = 0 $ for every $ j > l $, then it is not difficult to show that the value of the function depends only on $ \{ A \cap \{ 1, \dots, l \} : A \in \mathcal{A} \} $ (an analogous statement in the dual setting was discussed in \cite{ku}, see \cite[Fact~3.4]{ku}). For this reason, $ K(\mathcal{P}(\mathbb{N})) $ can be decomposed into finitely many clopen sets on which the function is constant.
\end{proof}

It is now quite easy to prove that the distances to which the Kadets distance is reducible are not Borel. We use the following classical result that can be found e.g. in \cite[(27.4)]{Ke}.

\begin{thm}[Hurewicz] \label{thmhur}
The set
$$ \mathfrak{H} = \Big\{ \mathcal{A} \in K(\mathcal{P}(\mathbb{N})) : \textnormal{$ \mathcal{A} $ contains an infinite set} \Big\} $$
is a complete analytic subset of $ K(\mathcal{P}(\mathbb{N})) $. In particular, it is not Borel.
\end{thm}

\begin{proof}[Proof of Theorem~\ref{thm:distancenotBorel}]
Let $ f : \Banach \to P $ be a Borel uniform embedding. We can find $ \eta > 0 $ and $ \varepsilon > 0 $ such that
$$ \rho(f(Y), f(Z)) < \eta \Rightarrow \rho_{K}(Y, Z) < \frac{1}{8} $$
and
$$ \rho_{K}(Y, Z) \leq \varepsilon \Rightarrow \rho(f(Y), f(Z)) < \eta $$
for all $ Y, Z \in \Banach $. Let $ X $ be as in Proposition~\ref{propnonBoreldist} and let $ \mathfrak{S} $ be a mapping provided for $ \varepsilon $. We claim that the set
$$ \Phi = \{ p \in P : \rho(p, f(X)) < \eta \} $$
is not Borel, and thus that the function $ \rho(f(X), \cdot) $ is not Borel.

Due to Theorem~\ref{thmhur}, it is sufficient to realize that $ (f \circ \mathfrak{S})^{-1}(\Phi) = \mathfrak{H} $, i.e.,
$$ \mathcal{A} \in \mathfrak{H} \quad \Leftrightarrow \quad f(\mathfrak{S}(\mathcal{A})) \in \Phi. $$
If $ \mathcal{A} \in K(\mathcal{P}(\mathbb{N})) $ contains an infinite set, then $ \rho_{K}(\mathfrak{S}(\mathcal{A}), X) \leq \varepsilon $, and so $ \rho(f(\mathfrak{S}(\mathcal{A})), f(X)) < \eta $. If $ \mathcal{A} \in K(\mathcal{P}(\mathbb{N})) $ consists of finite sets only, then $ \rho_{K}(\mathfrak{S}(\mathcal{A}), X) \geq \frac{1}{8} $, and so $ \rho(f(\mathfrak{S}(\mathcal{A})), f(X)) \geq \eta $.
\end{proof}
The following corollary then immediately follows from Theorem~\ref{thm:intro1}. In particular, it answers in negative Question 8.4 from \cite{BYDNT}.
\begin{cor}
Let $\rho$ be any pseudometric from the following list: $\rho_{GH}$, $\rho_{GH}\upharpoonright \Met_p$, $\rho_{GH}\upharpoonright \Met_p^q$, $\rho_{GH}^\Banach$, $\rho_K$, $\rho_L$, $\rho_N$, $\rho_U$, $\rho_{BM}$. Then there exists $A$ from the domain of $\rho$ so that the function $\rho(A,\cdot)$ is not Borel. 
\end{cor}

\section{Concluding remarks and open problems}\label{section:problems}
Regarding the notion of Borel-uniformly continuous reducibility, here and in the complementary article \cite{CDKpart2} we have focused mainly on the positive results. Any rich theory should however contain also the negative ones. Within the standard theory of definable equivalence relations, it is often the negative results, results demonstrating that some equivalence relations are not reducible to some other ones, that form the most interesting and challenging part of the theory. The Kechris and Louveau's result of non-reducibility of $E_1$ (\cite{KeLo97}), which we generalized in this paper, is an example. Hjorth's theory of turbulence is another main example, see \cite{Hjo}. For some pseudometrics it is clear that they do not reduce to each other for trivial reasons, e.g. the Gromov-Hausdorff distance to the cut distance on graphons defined in the section on orbit pseudometrics, as one is analytic non-Borel, while the other is Borel. Some more interesting non-reducibility results would be welcome.
\begin{problem}
Find some `natural pseudometric from functional analysis or metric geometry' that is not bi-reducible with the Gromov-Hausdorff distance.
\end{problem}

Note that, by our results, it would be sufficient to find such natural pseudometric which is not Borel-uniformly continuous reducible to a CTR orbit pseudometric.

It follows from the result of Zielinski in \cite{Zie}, that the homeomorphism relation on compact metrizable spaces is bi-reducible with the universal orbit equivalence relation, and from the results of Amir (\cite{Amir}) and of Dutrieux and Kalton (\cite{DuKa}) that the equivalences $E_{\rho_{GH}}$ and $E_{\rho_{BM}}$ are above the universal orbit equivalence in the sense of Borel reducibility. More thorough discussion about this fact is in \cite[Remark 8.5]{BYDNT}. By our results from the complementary paper \cite{CDKpart2} this is also true for $E_\rho$, where $\rho$ is any pseudometric from the set $\{\rho_{GH},\rho_{GH}\upharpoonright \Met_p, \rho_{GH}\upharpoonright \Met^q, \rho_{GH}\upharpoonright \Met_p^q, \rho_{GH}^\Banach,\rho_K,\rho_{HL},\rho_N, \rho_L, \rho_L\upharpoonright \Met_p^q, \rho_U, \rho_{BM}\}$. By our further results, all these equivalences $E_\rho$ (except $E_{\rho_U}$ for which we do not know the answer) have Borel classes. This suggests the main open question, already stated in \cite{BYDNT} for $\rho_{GH}$ and $\rho_K$.
\begin{question}
Are the equivalence relations $E_\rho$, where $\rho$ is from the list above, Borel reducible to an orbit equivalence relation?
\end{question}
If the answer is negative, these relations would form an interesting rather unexplored class of analytic equivalence relations with Borel classes. Moreover, it would shed light on some long standing conjectures about the Borel equivalence relation $E_1$, as this equivalence is not Borel reducible to $E_\rho$, where $\rho$ is as above. We refer the reader to Section \ref{section:orbitPseudometrics}, where we discuss these issues related to $E_1$.

Note however that even if the answer were positive, our particular reducibility results would still be of interest, as we work with a quantitative notion of reducibility (the Borel-uniformly continuous reducibility) that is stronger than the standard Borel reducibility.

\medskip
Clemens, Gao and Kechris prove in \cite{CleGaoKe} (see also \cite{GaoKe}) that the isometry relation on the Effros-Borel space $F(\mathbb{U})$, where $\mathbb{U}$ is the Urysohn universal metric space, is Borel bi-reducible with the orbit equivalence relation induced by the canonical action of $\mathrm{Iso}(\mathbb{U})$ on $F(\mathbb{U})$. One may ask if there is a continuous version of this result.
\begin{question}
Let $d$ be the Hausdorff distance on $F(\mathbb{U})$, $\mathrm{Iso}(\mathbb{U})\curvearrowright F(\mathbb{U})$ the canonical action and $\rho$ the corresponding orbit pseudometric (see Section \ref{section:pseudometrics} for a definition). Is $\rho_{GH}$ Borel-uniformly continuous bi-reducible with $\rho$? Or does at least one of the inequalities $\rho\leq_{B,u} \rho_{GH}$, $\rho_{GH}\leq_{B,u} \rho$ hold?
\end{question}

We have proved that there exists a universal analytic pseudometric $\rho$ with respect to Borel-uniformly continuous reducibility (even Borel-isometric reducibility) in Theorem \ref{thm:universal_pseudometric}. The corresponding equivalence relation is clearly the complete analytic equivalence relation. Several natural equivalence relations have been shown to be bireducible with the complete analytic one, e.g. the bi-Lipschitz homeomorphism of Polish metric spaces, or linear isomorphism of separable Banach spaces, see \cite{FLR}.
\begin{problem}
Is there a natural pseudometric that is bi-reducible with the universal analytic pseudometric $\rho$?
\end{problem}

\appendix

\section{Borelness of equivalence classes}\label{sectionGames}
The following surprising result was proved in \cite{BYDNT} using the methods of infinitary continuous logic.

\begin{thm}[see Corollary 8.3 and Theorem 8.9 in \cite{BYDNT}]\label{thm:fromBYDNT}
The equivalence classes of $E_{\rho_{GH}}$ and $E_{\rho_K}$ are Borel.
\end{thm}
It shows that these equivalences are somehow close to orbit equivalence relations (indeed, it is still open whether they are bireducible with the universal orbit equivalence): although they are analytic, their equivalence classes are Borel.

It is not clear how to use these methods for other distances though. The aim of this section is therefore to develop a complementary approach to that one from \cite{BYDNT} using games, which also provides alternative means how to prove Theorem \ref{thm:fromBYDNT}.
For the reader who is familiar with techniques of model theory we mention that the games we use here can be viewed as generalizations of the Ehrenfeucht-Fra\" iss\' e game to the metric setting. Recall that if $\mathcal{C}$ is some class of countable structures, then for $M,N\in\mathcal{C}$ we have $M\cong N$ if and only if Player II has a winning strategy in an appropriate version of Ehrenfeucht-Fra\" iss\' e game for $M$ and $N$ of length $\omega$. Analogously, for each countable ordinal $\alpha$ one can define a variant of the Ehrenfeucht-Fra\" iss\' e game played with $M$ and $N$ with parameter $\alpha$ such that Player II has a winning strategy there if and only if $M\equiv_\alpha N$, where $\equiv_\alpha$ is a certain Borel approximation of $\cong$.

\medskip

Let $\mathcal{X}$ be a standard Borel space and $X$ a countable set. A typical example of $\mathcal{X}$ is a class of separable metric structures (e.g. Polish metric spaces or separable Banach spaces) that can be described as completion of some countable metric structures with a countable set $X$ as an underlying set (more precisely, we assume in this case that $\mathcal{X}$ can be identified with some Borel subset of $\Rea^{X\times X}$).

Denote by $\mathcal{C}$ the Polish space of all correspondences $\Corr\subseteq X\times X$, which is a $G_\delta$ subset of the Polish space $\mathcal{P}(X\times X)$. Let $f:\mathcal{C}\times \mathcal{X}\times \mathcal{X} \rightarrow [0,\infty]$ be a Borel function. 

Suppose that the function $\rho:\mathcal{X}\times \mathcal{X}\rightarrow [0,\infty]$, defined as $$\rho(d,p)=\inf_\Corr f(\Corr,d,p)$$ is a pseudometric on $\mathcal{X}$. For every two finite sets $E,F\subseteq X$ and $\Corr\subseteq X\times X$ let $\Corr^{E,F}=\{(x,y)\setsep x\Corr y,x\in E,y\in F\}$. By $\mathcal{C}^{E,F}$ we shall denote the finite set $\{\Corr^{E,F}\setsep \Corr\in\mathcal{C}\}$. Suppose that for every pair of finite sets $E,F\subseteq X$ there are functions $f^{E,F}: \mathcal{C}^{E,F}\times\mathcal{X}\times\mathcal{X}\rightarrow [0,\infty]$ satisfying the following five conditions.

\begin{enumerate}[leftmargin=0.5cm,itemindent=.5cm,start=1,label={(\arabic*)}]
\item {\it Monotonicity with respect to inclusion}, that is, 
	\[f^{E,F}(\Corr^{E,F},d,p)\leq f^{E',F'}((\Corr')^{E',F'},d,p)\]
    whenever finite sets $E,E',F,F'$ are such that $E\subseteq E'$ and $F\subseteq F'$, and $\Corr\subseteq \Corr'$.
\smallskip
\item {\it Continuity in upward unions}, that is $$f(\Corr,d,p)=\sup_{E,F} f^{E,F}(\Corr^{E,F},d,p)=\lim _{E,F} f^{E,F}(\Corr^{E,F},d,p),$$ where the limit is taken over pairs of finite sets that increase in inclusion and eventually cover $X$.
\smallskip
\item {\it Borelness}, that is, for every $E,F$ finite subsets of $X$, every $\Corr\in\mathcal{C}^{E,F}$ and every $d\in\mathcal{X}$, the mapping $\mathcal{X}\ni e\mapsto f^{E,F}(\Corr,d,e)$ is Borel.
\smallskip
\item {\it Symmetry}, that is $ f^{E,F}(\Corr^{E,F},d,p) = f^{F,E}((\Corr^{E,F})^{-1},p,d) $ for all finite subsets $E,F\subseteq X$, every $\Corr\in\mathcal{C}$ and $d,p\in \mathcal{X}$.
\smallskip
\item  {\it Transitivity}, that is, $$f^{E,G}((\Corr')^{F,G}\circ \Corr^{E,F},d,e)\leq f^{E,F}(\Corr^{E,F},d,p)+f^{F,G}((\Corr')^{F,G},p,e)$$ for all finite subsets $E,F,G\subseteq X$, $\Corr,\Corr'\in \mathcal{C}$ and $d,p,e\in\mathcal{X}$.
\end{enumerate}

\medskip

\noindent {\bf Examples}

\smallskip

\begin{enumerate}[leftmargin=0cm,itemindent=.5cm,start=1,label={\bfseries \arabic*. }]
\item {\bf Gromov-Hausdorff distance.} For a correspondence $\Corr\subseteq \Nat\times\Nat$ and two metrics $d,p\in\Met$ we set $$f(\Corr,d,p)=\sup\big\{|d(n,m)-p(n',m')|/2\setsep n\Corr n', m\Corr m'\big\}.$$ For finite sets $E,F\subseteq \Nat$, the function $f^{E,F}$ is defined analogously. Using Fact \ref{fact:GHbyCorrespondences} it is easy to check that the functions $f,f^{E,F}$ have the desired properties and $\rho_{GH}$ is defined using them.

\medskip

\item {\bf Kadets distance.} For a correspondence $\Corr\subseteq V\times V$ and two norms $\Norm_X,\Norm_Y\in\Banach$, we define
\[\begin{split}
f(\Corr,\Norm_X,\Norm_Y)=\sup\Bigg\{ & \frac{1}{n}\bigg| \Big\|\sum_{i\leq n} x_i\Big\|_X-\Big\|\sum_{i\leq n} y_i\Big\|_Y\bigg| \setsep\\
& \forall i\leq n : \; \big(x_i\Corr y_i \text{ or } x_i(-\Corr) y_i\big), \, \| x_i\|_X \leq 1, \, \| y_i\|_Y \leq 1 \Bigg\}
\end{split}\]
if $\|x\|_X\leq 1 \Leftrightarrow \|y\|_Y\leq 1$ for all $ x, y $ with $x\Corr y$; otherwise we set
$$f(\Corr,\Norm_X,\Norm_Y)=\infty.$$
Functions $f^{E,F}$ are defined analogously. It is not very difficult, even though tedious, to prove using \cite[Lemma 17]{CDKpart2} that $ \rho = \rho_{K}$ (we leave the details as an exercise). Then it is easy to check that functions $f$, $f^{E,F}$ have the desired properties.

\medskip

\item {\bf Banach-Mazur distance.} For a correspondence $\Corr\subseteq V\times V$ and two norms $\Norm_X,\Norm_Y\in\Banach$ we define 
$$f(\Corr,\Norm_X,\Norm_Y)=\infty$$ if $\Corr$ does not extend to a graph of a surjective linear isomorphism $T:X\to Y$; otherwise we set
$$f(\Corr,\Norm_X,\Norm_Y)=\log\|T\|+\log\|T^{-1}\|,$$ if $\Corr$ extends to a graph of such an operator $T$. Using \cite[Lemma 23]{CDKpart2}, we check that indeed $\rho_{BM}(\Norm_X,\Norm_Y)=\inf\{f(\Corr,\Norm_X,\Norm_Y)\setsep \Corr\text{ is a correspondence on }V\times V\}$. We define $f^{E,F}(\Corr^{E,F},\Norm_X,\Norm_Y)$ to be the number $\log \|T\| + \log \|T^{-1}\|$ where $T:\Span \operatorname{Dom}(\Corr^{E,F})\to \Span \operatorname{Rng}(\Corr^{E,F})$ is the unique linear operator whose graph extends the relation $\Corr^{E,F}$ provided it exists. Otherwise, we set $f^{E,F}(\Corr^{E,F},\Norm_X,\Norm_Y)=\infty$. It is easy to check that functions $f$, $f^{E,F}$ have the desired properties.
\end{enumerate}

\begin{remark}
Another example considered here may involve $ X = \mathbb{N} $ and \emph{a logic} action of $S_\infty$. That is, a canonical action of $S_\infty$ on a space of the form $\prod_{i\in I} 2^{\Nat^{n_i}}$. The interpretation is as follows. Let $L=\{R_i\setsep i\in I\}$ be a countable relational language, where each $R_i$ is a relational symbol of arity $n_i\in\Nat$. Each element $x\in X_L=\prod_{i\in I} 2^{\Nat^{n_i}}$ corresponds to a relational structure $A_x$ with domain $\Nat$ such that the tuple $k_1,\ldots,k_{n_i}\in\Nat$ satisfies the relation $R_i^{A_x}$ if and only if $x(i)(k_1,\ldots,k_{n_i})=1$. We refer the reader to \cite[Chapter 3.6]{gao} for details. Note there that logic actions form the canonical Borel $S_\infty$-spaces (\cite[Theorem 3.6.1]{gao}).

Suppose moreover that the language $L$ is finite and that $X_L$ is a \emph{topometric space} as defined by Ben Yaacov in \cite{BY}. That is, $X_L$ is equipped with a metric $d$ which refines the compact Polish topology of $X_L$ and which is lower semi-continuous with respect to this Polish topology. Suppose that the action of $ S_\infty$ is by isometries. One may then define
$$ f(\Corr, x, y) = \left\{\begin{array}{ll} d(\pi \cdot x, y), & \quad \Corr = \pi \in S_{\infty}, \\
\infty, & \quad \Corr \notin S_{\infty}. \\
\end{array} \right. $$
Then the corresponding function $ \rho $ is nothing but the orbit pseudometric $ \rho_{S_{\infty}, d} $, which is moreover CTR provided that $d$ is complete. The continuity of the action and especially the lower semi-continuity of $d$ allow to define $ f^{E, F}(\Corr^{E,F},x,y)$ naturally as $$\inf_{\Corr'\in S_{E,F}}\inf_{x'\in E_x, y'\in F_y} f(\Corr',x',y'),$$ where $S_{E,F}$ is the open set $\{\Corr'\in \mathcal{C}\setsep (\Corr')^{E,F}=\Corr^{E,F}\}$; $E_x$, resp. $F_y$ is the neighborhood of $x$, resp. of $y$ determined by $E$, $F$ respectively. These last two neighborhoods are open by our assumption that $L$ is finite. We do not know if these conditions suffice to prove that the functions $f^{E,F}$ are Borel, or further restrictions on $d$ are necessary.
\end{remark}

\medskip

Fix now $d,p\in\mathcal{X}$, two finite tuples $\bar x$ and $\bar y$ from $X$ of the same length (the first is supposed to be from $(X,d)$, the second from $(X,p)$), and some $\varepsilon>0$. Let $\game(d\bar x,p\bar y,\varepsilon)$ denote a game of two players I and II. At the $i$-th step Player I chooses either a point $m_i\in X$ that is supposed to extend the tuple $\bar x m_1\ldots m_{i-1}$, or a point $n_i\in X$, that is supposed to extend the tuple $\bar y n_1\ldots n_{i-1}$. In the former case, Player II responds by playing an element $n_i\in X$ extending the tuple $\bar y n_1\ldots n_{i-1}$, or in the latter case a point $m_i\in X$ extending the tuple $\bar x m_1\ldots m_{i-1}$. The game has countably many steps in which the players produce infinite sequences $(m_i)_i$ and $(n_i)_i$ which define a relation $\Corr$ where $x_j\Corr y_j$, for $j\leq|\bar x|$, and $m_i\Corr n_i$, for $i\in\Nat$. Note that this is not in general a bijection as the players are allowed to repeat the elements. As $\Corr\subseteq X\times X$, one may see $\Corr$ as a correspondence between certain subsets $X_1,X_2\subseteq X$. By $\Corr_i$, for $i\in\Nat$, we define the subset of $X\times X$ given by the tuples $\bar x m_1\ldots m_i$ and $\bar y n_1\ldots n_i$. Analogously, we use the notation $f^i$ for the function $f^{\bar x m_1\ldots m_i,\bar yn_1\ldots n_i}$. By $f^0$ we mean the function $f^{\bar x,\bar y}$ and by $\Corr_0$ the subset of $X\times X$ given by the tuples $\bar x$ and $\bar y$.

At the end Player II wins if and only if $f^i(\Corr_i,d,p)<\varepsilon$, for all $i\in\Nat$.

If the tuples $\bar x$, $\bar y$ are empty, we denote the game just by $\game(d,p,\varepsilon)$.
\begin{remark}
It may well happen that $d=p$. Since we still need to distinguish between these two copies during the game (as the role of $d$ and $p$ is clearly not symmetric), we say that Player I in her first move either chooses a point $m_1\in X$ which extends $\bar x$, or chooses a point $n_1\in X$ which extends $\bar y$. After $i-1$-many steps when the players have produced tuples $\bar x m_1\ldots m_{i-1}$, $\bar y n_1\ldots n_{i-1}$, Player I either chooses $m_i\in X$ extending $\bar x m_1\ldots m_{i-1}$, or chooses a point $n_i\in X$ extending $\bar y n_1\ldots n_{i-1}$.
\end{remark}
\begin{lemma}\label{lem:DistanceZeroEquivalenceGames}
For every $d,p\in \mathcal{X}$, we have $\rho(d,p)=0$ if and only if for every $\varepsilon>0$ Player II has a winning strategy in $\game(d,p,\varepsilon)$.
\end{lemma}
\begin{proof}
Fix $d$ and $p$ from $\mathcal{X}$. Suppose that $\rho(d,p)=0$. Fix some $\varepsilon>0$. By the assumption there exists a correspondence $\Corr\subseteq X\times X$ such that $f(\Corr,d,p)<\varepsilon$. Now Player II can use $\Corr$ as his strategy. That is, if Player I plays some $m_i\in X$, then Player II responds by playing some $n_i\in X$ such that $m_i\Corr n_i$; or vice versa. It is clear that this is a winning strategy.

Conversely, suppose that Player II has a winning strategy in $\game(d,p,\varepsilon)$ for every $\varepsilon>0$. Fix some $\varepsilon>0$. Player I can play so that $\bigcup_i \{m_i\}=\bigcup_i \{n_i\}=X$. At the end the players produce a correspondence $\Corr\subseteq X\times X$ and by our assumptions we get
\[
f(\Corr,d,p)=\sup_i f^i(\Corr_i,d,p)\leq\varepsilon,
\]
thus $\rho(d,p)\leq f(\Corr,d,p)\leq\varepsilon$.
\end{proof}

Let $\alpha<\omega_1$ be now a countable ordinal, $\varepsilon > 0$ and $\bar x$ and $\bar y$ tuples of the same length from $X$. By $\game(d\bar x,p\bar y,\varepsilon,\alpha)$ we denote a game which is similar in its rules to $\game(d\bar x,p\bar y,\varepsilon)$, however in the first step Player I moreover chooses an ordinal $\alpha_1<\alpha$. In the $i$-th step, Player I chooses moreover an ordinal $\alpha_i<\alpha_{i-1}<\ldots<\alpha_1<\alpha$. The length of each play is finite, where the last step is when Player I chooses $0$ as an ordinal.  Analogously as above, Player II wins if $f^k (\Corr_k,d,p)<\varepsilon$, where $k\in \Nat$ is such that $\alpha_k=0$, and $\Corr_k$ is the correspondence produced by the players at the end of the game. If $\alpha=0$, then the game is decided in the very beginning and we set that Player II wins if $f^0(\Corr_0,d,p)<\varepsilon$.

\medskip
Let $\varepsilon>0$, $\alpha$ be a countable ordinal, and $\bar x$ and $\bar y$ be tuples of the same length. For every $(X,d)\in \mathcal{X}$, denote by $E(d,\bar x,\bar y,\varepsilon,\alpha)$ the set of all $p\in \mathcal{X}$ such that Player II has a winning strategy in the game $\game(d\bar x,p\bar y,\varepsilon,\alpha)$.
Again, if the tuples are empty, we write just $E(d,\varepsilon,\alpha)$ instead of $E(d,\emptyset,\emptyset,\varepsilon,\alpha)$.
\begin{lemma}\label{lem:GamesAreBorel}
$E(d,\bar x,\bar y,\varepsilon,\alpha)$ is Borel.
\end{lemma}
\begin{proof}
We shall prove it by induction on $\alpha$. Suppose that $\alpha=0$. Then the game is decided from the beginning and Player II wins if $f^0(\Corr_0,d,p)<\varepsilon$, where $\Corr_0$ is the correspondence given by the tuples $\bar x$ and $\bar y$. That is by definition a Borel condition, so $E(d,\bar x,\bar y,\varepsilon,0)$ is Borel.

Now suppose that $\alpha>0$ and we have checked that $E(d,\bar u,\bar v,\varepsilon,\beta)$ is Borel for all tuples $\bar u$ and $\bar v$, and $\beta < \alpha$.

If $\alpha$ is limit, then $E(d,\bar x,\bar y,\varepsilon,\alpha)$ is just $\bigcap_{\beta<\alpha} E(d,\bar x,\bar y,\varepsilon,\beta)$ which is Borel by assumption. So suppose that $\alpha=\beta+1$ for some $\beta$. Then by definition
\[\begin{split}
E(d,\bar x,\bar y,\varepsilon,\alpha) = & 
\left(\bigcap_{m\in X}\bigcup_{n\in X} E(d,\bar x m,\bar y n,\varepsilon,\beta)\right)\cap\\
 & \cap \left( \bigcap_{n\in X} \bigcup_{m\in X} E(d,\bar x m,\bar y n,\varepsilon,\beta)
\right)
\end{split}\]
which is Borel.
\end{proof}

\begin{lemma}\label{lem:tranzitivniHra}Let $d,p,e\in\mathcal{X}$ and let $\bar x$, $\bar y$, and $\bar z$ be tuples from $X$ of the same length. Let $\alpha$ be a countable ordinal, $\varepsilon > 0$ and $\varepsilon'>0$.
\begin{enumerate}
\item If Player II has a winning strategy in $\game(d\bar x,p\bar y,\varepsilon,\alpha)$ and in $\game(p\bar y, e\bar z, \varepsilon')$ then he also has a winning strategy in $\game(d\bar x,e\bar z,\varepsilon + \varepsilon',\alpha)$.
\item If Player II has a winning strategy in $\game(d\bar x, e\bar z, \varepsilon, \alpha)$ and in $\game(p\bar y, e\bar z, \varepsilon')$ then he also has a winning strategy in $\game(d\bar x,p\bar y,\varepsilon+\varepsilon',\alpha)$.
\end{enumerate}
\end{lemma}
\begin{proof}We only prove (1); (2) is proved analogously.

We shall prove it by induction on $\alpha$. It is clear that the statement holds for $\alpha = 0$. Now suppose that $\alpha > 0$ and the statement of the lemma holds for every $\beta < \alpha$. In order to shorten the description of the proof, let us denote by $\game_1$, $\game_2$ and $\game_3$ the games $\game(d\bar x,p\bar y,\varepsilon,\alpha)$, $\game(p\bar y, e\bar z, \varepsilon')$ and $\game(d\bar x,e\bar z,\varepsilon + \varepsilon',\alpha)$, respectively.

Suppose that Player I in $\game_3$ plays some ordinal $\alpha_1 < \alpha$ and a point $m_1\in X$ extending $\bar x$. Then we play the same move in game $\game_1$, and let Player II use his winning strategy in this game and pick a point $n_1 \in X$ such that Player II has a winning strategy in $\game(d\bar x m_1,p\bar y n_1,\varepsilon,\alpha_1)$.

Further, let Player I play $n_1$ in the game $\game_2$, and let Player II use his winning strategy in this game and pick a point $w_1\in X$ such that Player II has a winning strategy in $\game(p\bar y n_1,e\bar z w_1,\varepsilon')$. By the inductive assumption, Player II has a winning strategy in $\game(d\bar x m_1,e\bar z w_1,\varepsilon + \varepsilon',\alpha_1)$; hence we let Player II play $w_1$ as his response in the game $\game_3$. If Player I plays in her first move a point extending $\bar z$, then we proceed analogously.
\end{proof}

\begin{defin}
Fix $d\in\mathcal{X}$. For every $\varepsilon > 0$ and two tuples $\bar x$ and $\bar y$ from $X$ of the same length let $\alpha(\bar x,\bar y,\varepsilon)$ be the least ordinal $\alpha$ such that Player II does \underline{not} have a winning strategy in the game $\game(d\bar x,d\bar y,\varepsilon,\alpha)$. If Player II has a winning strategy in $\game(d\bar x,d\bar y,\varepsilon,\alpha)$ for every $\alpha<\omega_1$, then we set $\alpha(\bar x,\bar y,\varepsilon)=-1$.

We define a \emph{Scott rank} of $d$, $\alpha_d$ in symbols, to be\\ $\sup\{\alpha(\bar x,\bar y,\varepsilon)\setsep (\bar x,\bar y)\text{ appropriate elements of the same length},\varepsilon>0\}$.
\end{defin}
\begin{lemma}\label{lem:ScottRank}
Let $\varepsilon > 0$ and $\bar x$, $\bar y$ be tuples of the same length from $X$.

If Player II has a winning strategy in $\game(d\bar x,d\bar y,\varepsilon,\alpha_d)$ then he also has a winning strategy in $\game(d\bar x,d\bar y,\varepsilon,\alpha_d+1)$.
\end{lemma}
\begin{proof}
If it were not true, then we would have $ \alpha(\bar x,\bar y,\varepsilon) = \alpha_{d} + 1 > \alpha_{d} \geq \alpha(\bar x,\bar y,\varepsilon)$, which is a contradiction.
\end{proof}

Our aim is now to define a set of those $p\in\mathcal{X}$ such that $\rho(d,p)=0$, for a fixed element $d\in\mathcal{X}$. In the following, for a subset $A\subseteq\mathcal{X}$ we denote by $A^c$ the complement $\mathcal{X}\setminus A$. Also, we agree that $X^0$ denotes an empty sequence. We set
\[\begin{split}
I_d = \bigcap_{\varepsilon\in\Rat^+}\Bigg( & E(d,\varepsilon,\alpha_d)\cap
 \bigcap_{\substack{n\in\Nat\cup\{0\}\\ \bar x,\bar y\in X^n}} \bigg( E^c(d,\bar x,\bar y,\varepsilon,\alpha_d)\cup\\
 & \qquad \bigcap_{\varepsilon'\in\Rat^+} E(d,\bar x,\bar y,\varepsilon+\varepsilon',\alpha_d+1) \bigg)\Bigg).
\end{split}\]
It follows from Lemma \ref{lem:GamesAreBorel} that $I_d$ is a Borel set. To translate the definition above into words, it says that $p\in I_d$ if and only if for every $\varepsilon>0$ we have that Player II has a winning strategy in the game $\game(d,p,\varepsilon,\alpha_d)$ and, for all tuples of the same length (possibly empty tuples) $\bar x$, $\bar y$,  
if Player II has a winning strategy in the game $\game(d\bar x,p\bar y,\varepsilon,\alpha_d)$ then for every $\varepsilon'>0$ he has a winning strategy in the game $\game(d\bar x,p\bar y,\varepsilon+\varepsilon',\alpha_d+1)$. 

\begin{thm}\label{thm:LipEquivalenceBorel}
Let $d\in\mathcal{X}$ be arbitrary. For every $p\in \mathcal{X}$ we have that $\rho(d,p)=0$ if and only if $p\in I_d$.
\end{thm}
\begin{proof}
Fix some $d\in\mathcal{X}$ and also pick some $p\in\mathcal{X}$. By Lemma \ref{lem:DistanceZeroEquivalenceGames} it suffices to check that $p\in I_d$ if and only if Player II has a winning strategy in $\game(d,p,\varepsilon)$ for every $\varepsilon>0$.

We first show the left-to-right implication. So we fix some $\varepsilon>0$. Since $p\in I_d$, by definition $p\in E(d,\varepsilon/2,\alpha_d)$. Then it follows, also from the definition of $I_d$, that in fact $p\in E(d,3\varepsilon/4,\alpha_d+1)$. We start playing the game $\game(d,p,\varepsilon)$, which in the sequel will be denoted just by $\game$, with players I and II, however on the side we will also play several auxiliary games. We let Player I play her turn in $\game$, which is, say, an element $m_1$. We consider an auxiliary game $\game(d,p,3\varepsilon/4,\alpha_d+1)$ and we force the first player there to copy her move from $\game$ and moreover play the ordinal $\alpha_d$. Since the second player has a winning strategy in $\game(d,p,3\varepsilon/4,\alpha_d+1)$, we let him use the strategy and then use his response, say a point $n_1$, as a move of Player II in the main game $\game$. Now by definition Player II has a winning strategy in the game $\game(d m_1, p n_1, 3\varepsilon/4,\alpha_d)$. However, by the definition of $I_d$ this immediately implies that he has a winning strategy also in $\game(d m_1, p n_1,7\varepsilon/8,\alpha_d+1)$. We again let Player I play her second turn in the main game $\game$, which is, say, a point $n_2$. We consider an auxiliary game $\game(d m_1,p n_1,7\varepsilon/8,\alpha_d+1)$ and we force the first player there to copy her move from $\game$, i.e. playing $n_2$, and moreover play the ordinal $\alpha_d$. Since the second player has a winning strategy in $\game(d m_1,p n_1,7\varepsilon/8,\alpha_d+1)$, we let him use his strategy, which is, say, a point $m_2$, and then we use this response in the main game $\game$. Now by definition Player II has a winning strategy in the game $\game(d m_1 m_2, p n_1 n_2, 7\varepsilon/8,\alpha_d)$. However, again by the definition of $I_d$ this implies that he actually has a winning strategy also in $\game(d m_1 m_2, p n_1 n_2, 15\varepsilon/16,\alpha_d+1)$.

It is now clear that by using the winning strategies from these auxiliary games we get some strategy for Player II in the main game $\game$ which is winning.

\medskip
We now prove the reverse implication. So we suppose that Player II has a winning strategy in the game $\game(d,p,\varepsilon)$ for every $\varepsilon>0$. We show that $p\in I_d$. We need to show that for every $\varepsilon>0$ and $\varepsilon'>0$ we have that Player II has a winning strategy in the game $\game(d,p,\varepsilon,\alpha_d)$ and that for all tuples of elements of the same length $\bar x$, $\bar y$, 
if he has a winning strategy in $\game(d\bar x,p\bar y,\varepsilon,\alpha_d)$ then he has also a winning strategy in $\game(d\bar x,p\bar y,\varepsilon+\varepsilon',\alpha_d+1)$.

The former is clear. Indeed, if Player II has a winning strategy in $\game(d,p,\varepsilon)$, which is our assumption, then he also has a winning strategy in $\game(d,p,\varepsilon,\alpha_d)$. So we must show the latter. Fix some tuples $\bar x,\bar y$ of the same length. We need to show that if Player II has a winning strategy in $\game(d\bar x,p\bar y, \varepsilon,\alpha_d)$, then for every $\varepsilon'>0$ he has a winning strategy also in $\game(d\bar x,p\bar y,\varepsilon+\varepsilon',\alpha_d+1)$.

Since $\rho(d,p)=0$, Player II has a winning strategy in the game $\game(p,d,\varepsilon'/2)$, so there exists a tuple $\bar z$ of elements of the same length as $\bar y$ such that Player II has a winning strategy in the game $\game(p\bar y,d\bar z, \varepsilon'/2)$. 

Since, by Lemma~\ref{lem:tranzitivniHra}, Player II has  a winning strategy in $\game(d \bar x,d\bar z,\varepsilon+\varepsilon'/2,\alpha_d)$, by Lemma \ref{lem:ScottRank} we get that Player II has a winning strategy in $\game(d \bar x,d\bar z,\varepsilon+\varepsilon'/2,\alpha_d+1)$. By Lemma~\ref{lem:tranzitivniHra} again, Player II has a winning strategy in the game $\game(d\bar x,p\bar y,\varepsilon+\varepsilon',\alpha_d+1)$.
\end{proof}
The following corollary immediately follows from Theorem \ref{thm:LipEquivalenceBorel} and the fact that $I_d$ is a Borel subset of $\mathcal{X}$.
\begin{cor}\label{cor:LipBorelClasses}
For every $d\in\mathcal{X}$ the set $\{p\in\mathcal{X}\setsep \rho(d,p)=0\}$ is Borel.
\end{cor}

\bibliographystyle{siam}
\bibliography{ref}

\end{document}